\documentclass{article}
\usepackage{amssymb,amsmath,amsthm,graphicx}
\usepackage{amsfonts}
\usepackage[usenames,dvipsnames,svgnames,table]{xcolor}
\usepackage{comment}
\usepackage{color}
\usepackage{pinlabel}
\usepackage{mathtools}
\usepackage{thmtools} 
\usepackage{hyperref}

\usepackage{cleveref}

\usepackage[sortcites=true, backend=biber, hyperref=true, firstinits=true, maxbibnames=99,]{biblatex}

\newcommand{\Addresses}{%
\par\bigskip\noindent
\textsc{Ryan Blair}\\
Department of Mathematics, California State University Long Beach, Long Beach, CA 90840, USA\\
Email: \href{mailto:ryan.blair@csulb.edu}{ryan.blair@csulb.edu}

\par\medskip\noindent
\textsc{Puttipong Pongtanapaisan}\\
Mathematics Field Group, Pitzer College, Claremont, CA 91711, USA\\
Email: \href{mailto:puttip@pitzer.edu}{puttip@pitzer.edu}

\par\medskip\noindent
\textsc{Christine E. Soteros}\\
Department of Mathematics and Statistics, University of Saskatchewan, Saskatoon, Canada\\
Email: \href{mailto:soteros@math.usask.ca}{soteros@math.usask.ca}
\par\bigskip
}
\declaretheorem[name=Theorem]{thm}
\usepackage{caption, subcaption} 
\usepackage{graphicx}
\graphicspath{ {./images/} }
\usepackage{tikz}
\usepackage{pinlabel}
\usetikzlibrary{knots}
\newtheorem{theorem}{Theorem}
\newtheorem{lemma}[theorem]{Lemma}
\newtheorem{prop}[theorem]{Proposition}

\theoremstyle{definition}
\newtheorem{example}{Example}
\newtheorem{definition}{Definition}
\newtheorem{remark}{Remark}

\newcommand{\tube}{\mathbb{T}}

\date{}

\title{Entanglement complexity of spanning pairs of lattice polygons}

\author{Ryan Blair\footnote{Email: \href{mailto:ryan.blair@csulb.edu}{ryan.blair@csulb.edu}} \;\ Puttipong Pongtanapaisan\footnote{Email: \href{mailto:puttip@pitzer.edu }{puttip@pitzer.edu}} \;\ Christine E. Soteros \footnote{Email: \href{mailto:soteros@math.usask.ca}{soteros@math.usask.ca}}}
\begin{document}
\maketitle

\begin{abstract} 
We study the entanglement complexity of a system consisting of two simple-closed curves (self-avoiding polygons) that span a lattice tube, referred to as a 2SAP. 2SAPs are of interest as the first known model of confined ring polymers where the linking probability goes to 1 exponentially with the size of the system.  Atapour et al proved this in 2010 by showing that all but exponentially few sufficiently large 2SAPs contain a pattern that guarantees the 2SAP is non-split, provided that the requisite pattern fits in the tube. This result was recently extended to all tubes sizes that admit non-trivial links. Here we develop and apply knot theory results to answer more general questions about the entanglement complexity of 2SAPs. 

We first extend the 1992 concept of a good measure of knot complexity to a good measure, $F$, of spanning-link complexity for $k$-component links. Using tangle products, we show, for example, that the more complex the prime knot decomposition of any component of a given link type, the greater its $F$-measure. We then prove that all but exponentially few size $m$ 2SAPs have $F$ complexity that grows at least linearly in $m$ as $m\to \infty$. We establish that good measures of knot complexity yield good measures of spanning-link complexity.
We also establish conditions whereby more general link invariants can yield good measures. In particular,
we establish that measures based on several classical invariants are good measures by our definition, eg bridge number, the number of $p$-colourings or splitting number.
 
Finally, we consider how the tube dimensions affect which links are embeddable as 2SAPs as well as geometric restrictions on the entanglement complexity of the embeddings. 
For example, we establish that there are two-component links that occur as 2SAPs in a given tube size only when one of the components is forced into a non-minimal bridge number conformation.

\end{abstract}

\parbox{5.5in} {\textsc{Keywords:} 

\smallskip
knot complexity, lattice polygons, link complexity, polymer confinement, nanochannel, tangle product, bridge number, splitting number, trunk}

\section{Introduction}
Single self-avoiding polygons (SAPs) on the simple cubic lattice ($\mathbb{Z}^3$) have been used to
model and study the self-entanglement of closed curves and ring polymers for over 35 years.
In particular, the Frisch-Wasserman-Delbruck (FWD) conjecture \cite{delbruck1962mathematical,frisch1961chemical} that ring polymers are more likely to be  knotted  as polymer length increases, was first proved using a lattice polygon model \cite{sumners1988knots,Pip89}. For this, it was proved that all but exponentially few sufficiently long
lattice polygons are knotted \cite{sumners1988knots}, ie the probability of the unknot decreases exponentially to zero with polygon length (number of lattice edges). Similar questions
regarding the entanglement complexity of multiple closed curves have also been addressed with
lattice SAP models, where multiple closed curves are modelled by mutually avoiding
SAPs.

For pairs of mutually avoiding cubic lattice SAPs with total length $n$ (sum of the lengths
of each SAP), if the two SAPs are constrained to have a pair of edges (one from each SAP)
within a fixed distance apart, then the exponential growth rate (with respect to $n$)
of non-split (linked) SAP pairs is equal to that of split (unlinked) SAP pairs \cite{orlandini1994random}. That is, unlike the single SAP case with knotting, we cannot say that all but exponentially few SAP pairs are linked,
even with a fixed distance constraint. Note that it is still possible that the linking probability
goes to one, but it will not do so exponentially fast (with respect to $n$).
The same has been shown to hold when confining pairs of SAPs to sublattices of $\mathbb{Z}^3$ such as slabs or rectangular tubes \cite{tesi1998topological}.  The $M$-slab is the sublattice of $\mathbb{Z}^3$ bounded by the two parallel $z$-planes $z=0$ and $z=M$ and the $N\times M$ tube or $(N,M)$-tube is the sublattice of the $M$-slab bounded by the two parallel $y$-planes $y=0$ and $y=N$. 
Thus even when confined to a lattice tube, the exponential growth rates for linked and unlinked SAP pairs are  equal \cite{tesi1998topological} (again assuming a pair of edges (one from each SAP) are close). However, if a pair of SAPs in the tube
are forced to both extend over the same volume (ie have the same $x$-span - see Figure \ref{fig:isa2sap}) then all but exponentially few sufficiently large spanning pairs of SAPs (called 2SAPs) in
the $N \times M$ tube are linked. This was proved first for tubes with $\min\{M,N\}\geq 1$ and $N + M \geq 4$ \cite{Atapour10} and was  recently extended to $\min\{M,N\}\geq 1$ and $N + M \geq 3$, ie to all tube sizes that admit non-trivial links \cite{Inprep2025}. Thus, as one might expect, if all edges of each polygon are forced to be close to the edges of the other polygon, non-trivial linking becomes highly probable.  Establishing these results for 2SAPs  required establishing and applying non-trivial results from knot theory.

\begin{figure}
\centering
\begin{subfigure}{0.4\textwidth}
    \includegraphics[width=\textwidth]{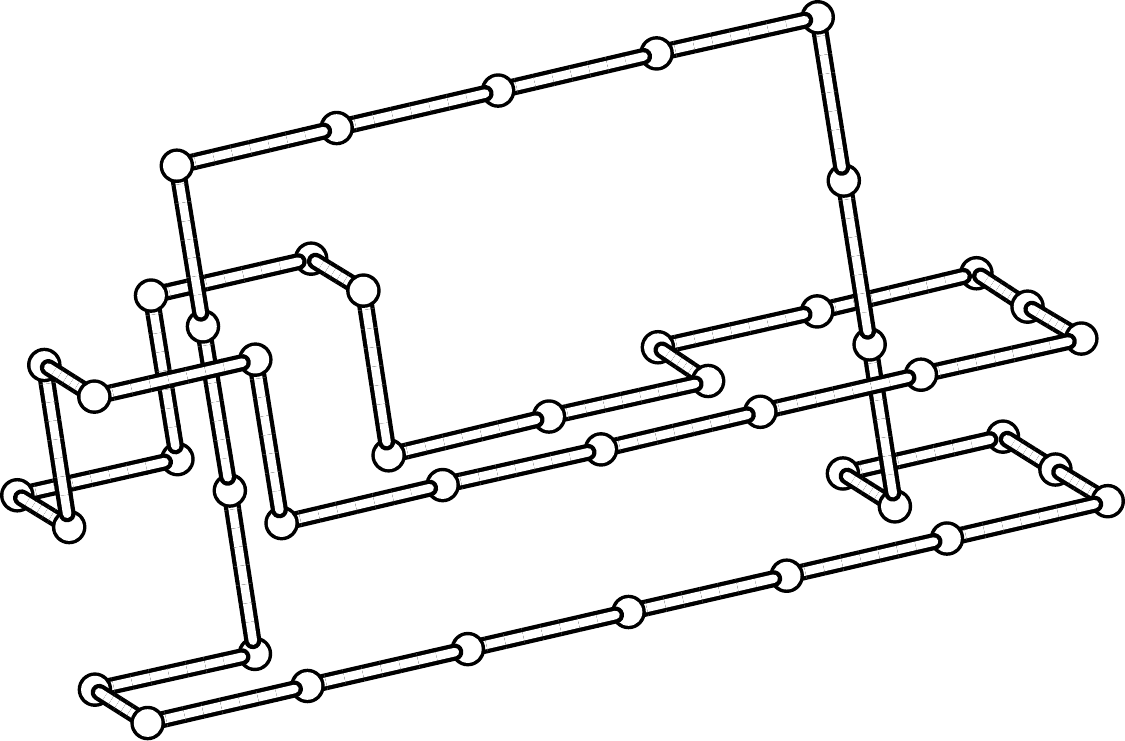}
    \caption{}
    \label{fig:isa2sap}
\end{subfigure}
\hfill
\begin{subfigure}{0.4\textwidth}
    \includegraphics[width=\textwidth]{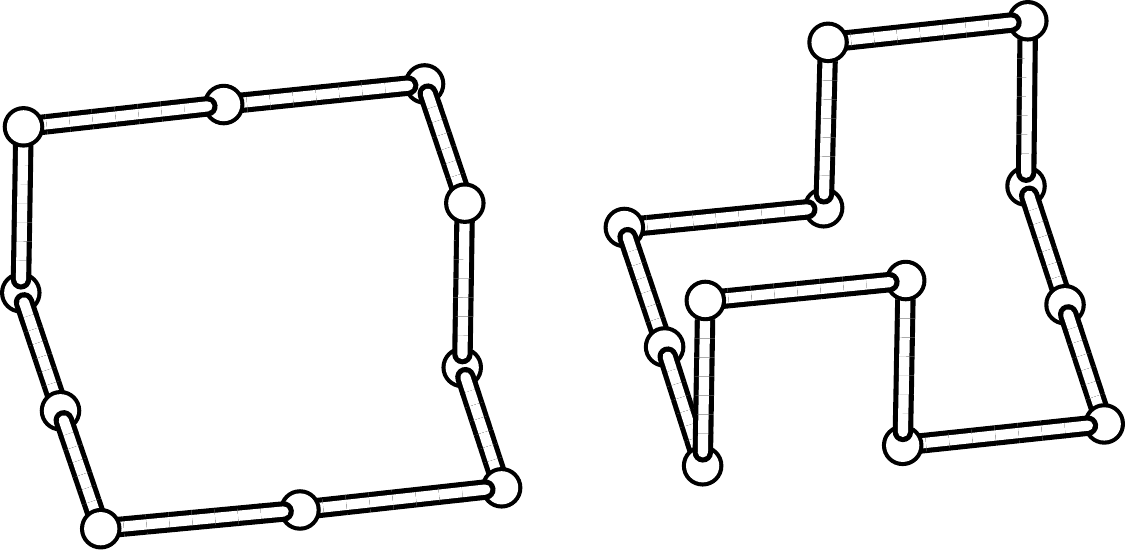}
    \caption{}
    \label{fig:nota2sap}
\end{subfigure}
\caption{(a) An embedding of a 2-component unlink in a $(2\times 3)$ lattice tube that is a 2SAP. In this case, the span of the 2SAP is 6. (b) An embedding of a 2-component unlink in a $(2\times 1)$ lattice tube that is not a 2SAP.}
\label{fig:figures}
\end{figure}

In this paper, we explore the entanglement complexity of  2SAPs further by developing and applying novel topological theorems and approaches.  We build on the concept of ``good measures of knot complexity"  defined   with respect to the connected sum operation \cite{SSW92} by introducing here ``good measures of spanning-link complexity" with respect to a concatenation operation defined  for two $k$-component links  as in Figure \ref{fig:tangledconcat} for $k=2$. 

In general, good measures for knot complexity  increase  with the complexity of the prime knot decomposition of the knot.  
Related to this, here we will use tangle products and  define good measures for spanning-link complexity to be ones that, for example, increase with the complexity of the prime knot decomposition of the components of the link. We establish that  good measures of knot complexity yield good measures of spanning-link complexity  and also establish conditions whereby more general link invariants can yield good measures. In particular,
we establish that measures based on several classical link invariants, eg  bridge number, the number of $p$-colorings and splitting number,  are good measures by our definition. {Combining these knot theory results with lattice combinatorics arguments, allows us to establish that all but exponentially few sufficiently large 2SAPs have high entanglement complexity, provided $\min\{M,N\}\geq 1$ and $N+M\geq 4$. 
Note that this does not include the $2\times 1$ tube since all 2SAPs in that tube are prime \cite{Inprep2025}.

\begin{figure}
\centering
\begin{subfigure}{0.4\textwidth}
    \includegraphics[width=\textwidth]{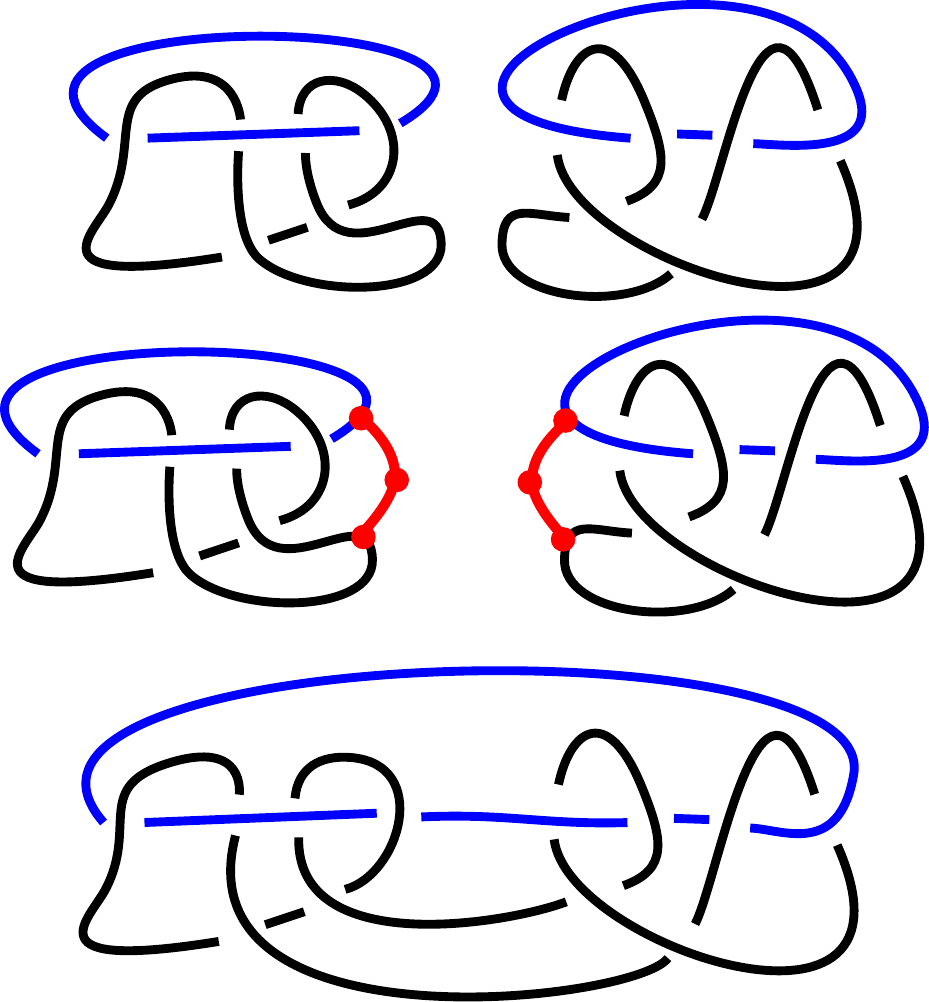}
    \caption{}
    \label{fig:tangledconcat}
\end{subfigure}
\hfill
\begin{subfigure}{0.4\textwidth}
    \includegraphics[width=\textwidth]{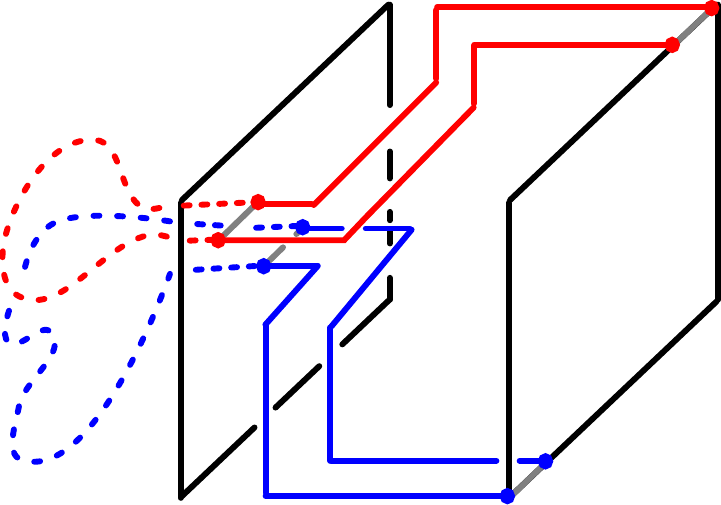}
    \caption{}
    \label{fig:kequals2}
\end{subfigure}        
\caption{(a) Two 2-component links (top) are concatenated together by component-wise connected sum operations  to create a new link (bottom). This concatenation operation is equivalent to a 2-strand tangle product operation (middle) as in Definition \ref{def:product}. This example also illustrates that  concatenating two unlinks can yield a non-split link. (b) Half of the concatenation algorithm described in Definition \ref{def:concat} for concatenating $k$SAPs, $k=2$: dashed lines represent a 2SAP $\theta_1$ and solid lines indicate half the edges added to concatenate to any other 2SAP. }
\label{fig:figures2}
\end{figure}

Another direction where knot theory is important for analyzing the complexity of lattice knots/links, is with regards to determining which knots/links are embeddable in a given sublattice and the minimum number of steps needed \cite{SSW92,ishihara2017bounds,Ishihara_2012,rechnitzer_rensburg,KIM2025109353,Scharein_2009}.  For polygons in the simple cubic lattice all knots are embeddable \cite{SSW92} and there are proofs regarding the minimal number of steps needed to create some knots \cite{Scharein_2009}.  Algorithms have also been used to obtain upper bounds for these minimal step numbers for knots and links in various lattices and for slab and tube sublattices of the simple cubic lattice \cite{rechnitzer_rensburg,ishihara2017bounds,Ishihara_2012}.  
In \cite{ishihara2017bounds}, the authors used the complexity called the trunk of knots \cite{ozawa2010waist} to completely characterize the knots and links that fit inside an $N\times M$ tube.

Also for tubes, in \cite{beaton2018characterising} the authors  characterized different 
conformations or ``modes" of a given prime knot pattern (2-string tangle) as either local (aka 1-filament) or non-local (aka 2-filament) depending, respectively, on whether or not the link-type of the numerator closure still contains the knot (see \cite[Figure 3]{beaton2018characterising} and also \cite{Suma17}). It was also shown in \cite{beaton2018characterising} that for small tube sizes, the non-local knot patterns are more likely than the local ones.

Here we introduce the concept of equal-height trunk to establish which 2-component links can occur as a 2SAP in an $N\times M$ tube. Further we explore the likelihood of different ``modes" of linking for 2SAPs.  For example, for a prime non-split link $L$ having each component an unknot, if it can occur as a 2SAP in a given tube, we consider whether it is possible to minimize the number of maxima in one component of the 2SAP.  The latter can have applications to studying the translocation of linked DNA through a nanopore \cite{COW17, COW20, RK23}. In particular, 
simplifying the  conformation of one of the unknotted components may prevent the translocation of the entire link, even through pores that would otherwise admit the linked DNA.

The main theorems proved in this paper are stated below. Full details of proofs are given in subsequent sections.

Using tangle products, we define good measures of spanning link complexity in Definition \ref{def:goodmeasure} of Section \ref{sec:linkGoodMeasures}. Roughly speaking, good measures of $k$-component spanning link complexity increase linearly with the number of occurrences of a given tangle (specified by what we will call a $k$-tangle triple) in a particular tangle decomposition, one which includes the concatenation operation defined for 2SAPs. 
We then establish the following general ways for obtaining good measures.  
Let $\mathcal{L}^k$  be the set of smooth ambient isotopy classes of $k$-component links.

\begin{restatable}{thm}{LinkMeasureToGoodMeasure} \label{thm:LinkMeasureToGoodMeasure}
 Suppose that $F:\bigcup_{i=1}^{\infty} \mathcal{L}^i\rightarrow [0,\infty)$ is a function defined on links of any number of components. Additionally, suppose that the following hold:
 
 \begin{enumerate}
\item if $U$ is an unlink, $F(U)=0$,
\item if a link $L=L'\cup U$ is a split link with splitting sphere separating $L'$ from $U$, then $F(L)=F(L')$,
\item the restriction $F:\mathcal{L}^1\rightarrow [0,\infty)$ is a good measure of knot complexity, and
\item for any knot $K$ and for any link $L$ with component knot $K$, $F(L)\geq F(K)$.
\end{enumerate}
 
Then the restriction 
$F:\mathcal{L}^k\rightarrow [0,\infty)$ is a good measure of $k$-component spanning link complexity for any $k\geq 1$.
\end{restatable}

In Section \ref{sec:linkGoodMeasures}, we establish that taking $F$ equal to the bridge number minus $k$, crossing number, unlinking number, or braid index minus $k$ each satisfy criteria 1-4 above and hence are good measures.  

\begin{restatable}{corollary}{goodmeasure}
\label{cor:thm1}
   If $F_1:\mathcal{L}^1\rightarrow [0,\infty)$ is a good measure of knot complexity, then $F:\cup_{i=1}^\infty \mathcal{L}^i\rightarrow [0,\infty)$ given by $F(L)=max_{K\subset L} F_1(K)$ where the maximum is taken over all components $K$ of $L$, is a good measure of $k$-component spanning link complexity when restricted to $k$-component links. 
\end{restatable}
Thus from the arguments in \cite[Theorem 3.3]{SSW92}, taking the component-wise maximum of each of the following knot invariants gives a good measure of $k$-component spanning link complexity: number of prime
factors, genus, bridge number minus one, span of any non-trivial Laurent knot
polynomial, log(order), crossing number, unknotting number, minor index, braid
index minus one.

Beyond these results, we  show that some measures that do no satisfy the conditions above are still good measures. 

\begin{figure}
  \labellist
\small
\pinlabel{$K$} at 120 122
\endlabellist
\centering
\begin{subfigure}{0.4\textwidth}
    \includegraphics[width=\textwidth]{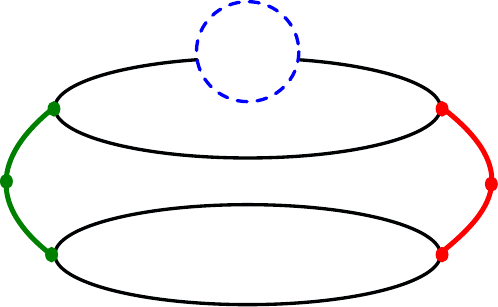}
    \caption{}
    \label{fig:L(K)separated}
\end{subfigure}
\hfill
\begin{subfigure}{0.4\textwidth}
    \includegraphics[width=\textwidth]{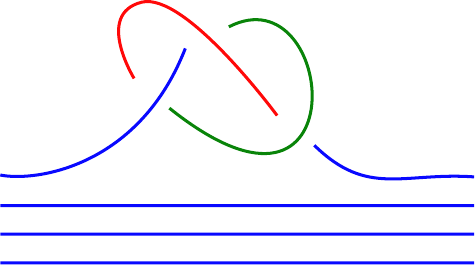}
    \caption{}
    \label{fig:tight}
\end{subfigure}
\caption{(a) The link $L(K)$ is the split union of $k$ knots (shown in black for $k=2$) where one is knot-type $K$ (localized in the dashed blue circle) and the others are the unknot ($0_1$). $L(K)$ is also used to represent the $k$-tangle triple $(L(K),G_1,G_2)$ where $G_1$ and $G_2$ are the $k$-star graphs shown in green and red respectively for the case $k=2$ (see Definition \ref{def:tangle} in Section \ref{sec:linkGoodMeasures}) (b) A pattern that increases some complexities of 2SAPs. Using our terminology, this is $L(3_1)$.}
\label{fig:figures3}
\end{figure}

\begin{restatable}
{thm}{moregoodmeasures}

Given $k \geq 1$, suppose that $F$ is a function $F:\mathcal{L}^k\rightarrow [0,\infty)$ satisfying:
\begin{itemize}
    \item[(1)] if $U\in \mathcal{L}^k$ is an unlink, $F(U)=0$; 
    \item[(2)] there exists a knot $K$ such that for any link $L\in \mathcal{L}^k$ containing $m$ factors having knot-type $K$, $F(L)\geq mF(L(K))$, where $L(K)$ is defined in Figure \ref{fig:L(K)separated}. 
\end{itemize}
Then, $F$ is a good measure of $k$-component spanning link complexity.\label{thm:moregoodmeasures}
\end{restatable}

In Section \ref{subsec:additional}, we establish that the genus, the number of prime factors, the span of the Alexander polynomial and the log of the number of $p$-colorings each satisfy the conditions of Theorem \ref{thm:moregoodmeasures} and hence are good measures. 

Although Theorem \ref{thm:LinkMeasureToGoodMeasure} and Theorem \ref{thm:moregoodmeasures} generate many examples of good measures of $k$-component spanning link complexity, there are other good measures that do not meet the hypotheses of either of these theorems, for example the splitting number. The splitting number of a link is defined to be the minimum number of crossing
changes between distinct components that are required to convert the link into a split link. In Proposition \ref{Prop:splittingGood} of Section \ref{subsec:additional}, we show that the splitting number is a good measure of 2-component spanning link complexity and that it cannot meet the hypotheses of either Theorem \ref{thm:LinkMeasureToGoodMeasure} or Theorem \ref{thm:moregoodmeasures}.

Having determined a range of good measures of spanning link complexity, we next consider the implications for 2SAPs in an $(N,M)$-tube $\tube=\tube_{N,M}$ where from here on we assume $M\geq N\geq 1$ and $M+N\geq 3$.  To determine which links are embeddable as a 2SAP in a given tube, we introduce a constrained version of the trunk complexity for a 2-component link $L=K_1\cup K_2$ which we call Equal-height-trunk (EH-trunk) (a precise definition is given in Section \ref{Sec:EHTrunk}). The modification of the definition of trunk is necessary for the study of 2SAPs due to the restrictions on the span of each component.  In Section \ref{Sec:EHTrunk}, we prove the following theorem.

\begin{restatable}{thm}{equalheight}
\label{thm:equalheight}
    A nontrivial 2-component link can be confined in a $M\times N$ tube in such a way that each component has the same span if and only if \text{EH-trunk}$(L) < (M+1)(N+1)$.
\end{restatable}

We also consider different ``modes" of embeddability for 2SAPs. We first generalize the notion of local/non-local (aka 1-filament or 2-filament) knot patterns for one component links from \cite{beaton2018characterising} to  2SAPs and obtain a version of \cite[Result $4^*$]{beaton2018characterising} using equal-height trunk. For this, we consider split 2-component links $L=0_1 \cup K$ and consider 2SAPs equivalent to $L$. Removing one edge from each component at each end of the 2SAP, yields a 2SAP pattern (see definitions in Section \ref{sec:preliminaries}).  If the pattern is isotopic as a 4-string tangle to Figure \ref{fig:tight} (where the trefoil is replaced by $K$) then we say the pattern is a \textit{local 2SAP knot pattern of $K$} and we say that the pattern corresponds to the 2-tangle-triple denoted by $L(K)$ (see Figure \ref{fig:L(K)}).  Otherwise we say it is a \textit{non-local 2SAP knot pattern of $K$}. We show that for meridionally small knots, all tubes that admit a non-local knot  pattern also admit a local version (see Lemma \ref{prop:fittinglocal}), however, for other kinds of knots a larger tube may be needed for the local than for the non-local knot patterns.  
The class of meridionally small knots includes 2-bridge knots and torus knots. Note that a knot is meridionally small if there exists no essential surface properly embedded in the knot exterior with meridional boundary, for more details see \cite{ozawa2010waist}.
We obtain the following corollary to Theorem \ref{thm:equalheight}.

\begin{restatable}{corollary}{equalheightplus}
\label{cor:equalheight}
For a split 2-component link $L=0_1\cup K$, \text{EH-trunk}$(L)=$\text{trunk}$(K)+2$. Furthermore, if $L$ embeds as a 2SAP in a  $(N\times M)$ tube (ie if $\text{EH-trunk}$(L)$<(N+1)(M+1)$), a local 2SAP knot pattern of $K$ embeds in a $((N+1)\times M)$ tube. For any meridionally small knot $K$, a local 2SAP knot pattern of $K$ also embeds in a  $(N\times M)$ tube.

\end{restatable}

Based on the pattern theorem for 2SAPs, in Section \ref{sec:2SAPcomplexity} we show that all but exponentially few 2SAPs have high entanglement complexity.  For this we consider two ways of measuring the {\it size} $m$ of a 2SAP: by {\it span} $s$ ($x$-span of the 2SAP)  or by {\it length} $n$ (total number of 2SAP edges). (Definitions are given in Section \ref{sec:preliminaries}.)  The resulting theorem is as follows:

\begin{restatable}{thm}{exprare}
\label{thm:exprare}
Let  $F$ be a good measure of 2-component spanning link complexity and let $L\in {\cal L}_T^2$ be a 2-tangle triple that satisfies condition (ii) for $F$ of the good measure Definition \ref{def:goodmeasure}. Suppose there exists a proper 2SAP pattern $P_L$  such that whenever a 2SAP  contains $j>0$ non-overlapping translates of $P_L$ it is equivalent as a link to a tangle product $L' = J_1\underline{*} L \underline{*} J_2 \underline{*} L \underline{*} J_3 \underline{*}...\underline{*}J_{j}\underline{*}L\underline{*}J_{j +1}$ for some choice of $J_1,...,J_{j+1}\in{\cal L}_T^2$. 

Consider any $\tube=\tube_{N,M}$  that admits 2SAPs  containing $P_L$. 

Fix the 2SAP size measure $m$ to be either  span or  length.  Then, there exists a positive integer $m_L$ such
that for sufficiently large $m > m_L$, all but exponentially few size  $m$ 2SAPs in $\tube$ have $F$-complexity
which exceeds $\displaystyle{F(L)(\frac{m}{2m_L} -1)}$.
\end{restatable}

Then in Proposition \ref{prop:patterns} we establish that an appropriate $P_L$ exists for splitting number and for any  good measure of 2-component spanning link complexity $F$ that satisfies either Theorem \ref{thm:LinkMeasureToGoodMeasure} or Theorem \ref{thm:moregoodmeasures}.
For the latter two cases, a local 2SAP knot pattern of a knot $K$ works as $P_L$ in the theorem. 
If $F$ satisfies Theorem \ref{thm:LinkMeasureToGoodMeasure},  $K$ can  be any knot, and hence taking $K=3_1$ will give for $\tube=\tube_{3,1}$ (or larger) that all but exponentially few size sufficiently large size $m$ 2SAPs in $\tube$ have $F$-complexity
which exceeds $\displaystyle{F(L(3_1))(\frac{m}{2m_L} -1)}$. In the case of splitting number, we further establish that a $P_L$ that works fits in $\tube=\tube_{3,1}$ (or larger) and is not a local knot pattern.

Note that we expect the previous result  can be extended in at least two directions, however, we do not present the details here.  First, based on the general pattern theorem \cite{Madras99}, the result should also hold for length $n$ 2SAPs subject to a stretching or compressing force as well as to Hamiltonian 2SAPs.  Second, we expect it can be generalized to $k$SAPs, $k$ mutually avoiding spanning SAPs in $\tube$.

For Corollary \ref{cor:equalheight}, we considered split 2-component links where one component is the unknot and the other a given knot. The result indicates that a larger tube size may be needed to embed a local knot pattern.    We next consider the more complicated situation of non-split links. Even when one of the components is the unknot, determining which tube admits an embedding with a given set of desired properties (for example one that achieves the minimum bridge number)  is more complicated.  For example, when studying invariants of multi-component links defined in terms of minimizing certain complexities over all diagrams, it is natural to ask whether it is always more efficient to calculate the invariant by minimizing the complexity of each knot component individually. For instance, if the complexity is defined as the minimum number of crossing changes required to trivialize a link, the authors of \cite{friedl2014specious} gave an example, a non-split 3-component link of an unknot with two trefoils having unlinking number 2, where the answer is yes. In other words, for the given example, changing crossings between components always results in a non-minimal unlinking sequence. On the other hand, 
Kodani showed that when the invariant is the bridge number and one component is the unknot, then there is a large class of non-split links where minimizing each component first is not  the most efficient way \cite{kodani2013new}. 

In this paper, we prove an analogue of Kodani's result for the EH-trunk. For this, we introduce a more constrained version of the  equal-height-trunk measure, denoted EH-trunk$_{K_1=j}(L)$, which fixes the trunk of the $K_1$ component to be $j$ (a more precise definition is given in Section \ref{Sec:EHTrunk}). This modification is necessary due to our interest in forcing one component to meet each cross-section of the tube in a prescribed number of points. For example, suppose that $L=K_1\cup K_2$ is a non-split 2-component link such that $K_1$ is the unknot. We are interested in finding the smallest lattice tube size in which $L$ can occur as a 2SAP. A naive way to do so is to fit $K_1$ first in the simplest way possible and then embed $K_2$ later. In this paper, we show that this is not always the most efficient way. We also discuss cases where trying to force one component to have a certain number of local maxima can cause the need for a larger tube size.
We achieve this by bounding EH-trunk$_{K_1=j}(L)$ in terms of the index $\sigma_2$ and the dual index $\omega_1$ for certain 2-component satellite links $L$. These indices arise naturally when we decompose the exterior of our satellite link along essential tori. Roughly speaking, we form a two component satellite link $L=K_1\cup K_2$ by taking two solid tori $V_1$ and $V_2$ and knotting them in 3-space. The knotted versions of $V_1$ and $V_2$ contain $K_1$ and $K_2$, respectively. We say $V_i$ is the companion of $K_i$. The index $\sigma_i$ of $K_i$ in $V_i$ is the minimal number of times that $K_i$ intersects a properly embedded meridian disk for $V_i$. If $V_1$ is unknotted, then $V'=\overline{S^3\setminus V_1}$ is a solid torus and the index of $V_2$ as a knot in $V'$ is the dual index $\omega_1$. See Section \ref{Sec:SatLinks} for additional details.

\begin{restatable}{thm}{restrictedtrunk} \label{thm1} Fix $j \geq \beta(K_1)$, where $\beta(K_1)$ is the bridge number of $K_1$ in $S^3.$ Suppose that the index $\sigma_1$ of $K_1$ in $V_1$ is strictly greater than $j$. Let $L=K_1\cup K_2$ be a 2-component satellite link, where $K_1$ and the companion of $K_1$ are each unknots. Then, EH-trunk$_{K_1=j}(L)\geq 2+2\omega_1\sigma_2$.
\end{restatable}

We finish the paper by giving several examples of links whose EH-trunk$(L)$ and EH-trunk$_{K_1=m}(L)$ can be very different. Consequently, for each such example $L$, to embed the link $L$ in a small tube size, each component of the link has to be embedded in a complicated way. Some of our examples are inspired by the works of Kodani \cite{kodani2013new}. The techniques used in our proofs, however, are fundamentally different from Kodani's since Kodani was trying to minimize a different complexity than us. In particular, Kodani was able to use techniques pioneered in \cite{schultens2003additivity} to characterize the saddle types of a taut essential torus completely. In contrast, we focus on a very specific (highest) saddle.

\section{Good measures of entanglement complexity}

\subsection{Preliminaries}
\label{sec:preliminaries}
In this paper $\mathbb{Z}^3$ is used to denote the simple cubic lattice consisting of vertices in $\mathbb{Z}^3$ and edges that join vertices unit distance apart. Recall that an embedding is a continuous map between topological spaces $f:X\rightarrow Y$ such that $f:X\rightarrow f(X)$ is a homeomorphism. The \textit{length} of an embedding $f:S^1\rightarrow \mathbb{Z}^3$ is the number of edges in $f(S^1)$. 
An embedding of $S^1$ in $\mathbb{Z}^3$ is called a self-avoiding polygon or SAP for short. 
An $N \times M$ \textit{tube} is the sublattice of $\mathbb{Z}^3$ induced by the vertex set $\{(x,y,z)\in \mathbb{Z}^3| x \geq 0, 0 \leq y \leq N, 0 \leq z \leq M \}$, we denote this tube by $\mathbb{T}_{N,M}$. For a given tube, the \textit{span} of an embedding of $S^1$ is the absolute difference between the smallest and largest $x$ value attained by any vertex in the image of the embedding. 

\begin{definition}
\label{def:kSAP}
A $k$SAP in a tube $\mathbb{T}_{N,M}$ is the image of an embedding $f:\bigcup_{i=1}^{k} S^1 \rightarrow \mathbb{T}_{N,M}$, which consists of $k$ disjoint self-avoiding polygons with the additional property that each component has the same span and each has a vertex in the $x=0$ plane.
\end{definition}

Consequently, the \textit{span} of a $k$SAP is the largest $x$ value attained by the $k$SAP. See Figure \ref{fig:isa2sap} for a 2SAP of span 6. See Figure \ref{fig:nota2sap} for a system of two self-avoiding polygons that is not a 2SAP.

Note that although different SAPs in a $k$SAP have the same span, they may not have the same number of edges. For a given $\tube=\tube_{N,M}$, define $\mathcal{Q}^k_{\tube}$ to be the set of $k$SAPs in $\tube$. Let $\mathcal{Q}^k_{\tube,s}(n)$ be the subset of $k$SAPs in $\mathcal{Q}_\tube$ with span-$s$ and length-$n$, and let $q^{(k)}_{\tube,s}(n) = |\mathcal{Q}^k_{\tube,s}(n)|$. Depending on the application, the   {\it size} of a $k$SAP can be considered to be either its span or its length.  Thus it is useful to define $q^{(k)}_{\tube,s}=q^{(k)}_{\tube,s}(\cdot) = \sum_n |\mathcal{Q}^k_{\tube,s}(n)|$, the number of span-$s$ $k$SAPs, and $q^{(k)}_\tube(n) = q^{(k)}_{\tube,\cdot}(n)=\sum_s |\mathcal{Q}^k_{\tube,s}(n)|$,
the number of length-$n$ $k$SAPs.

For SAPs and 2SAPs ($k=1,2$) concatenation and transfer-matrix arguments have been used to establish pattern theorems and the existence of exponential growth rates for the counts $q^{(k)}_{\tube,s}, q^{(k)}_\tube(n)$.    The concatenation operation needed for these arguments is defined next.  

\begin{definition}
\label{def:concat}
The concatenation of two $k$SAPS, $\theta_1\in \mathcal{Q}^k_{\tube,s_1}(n_1)$ and $\theta_2\in \mathcal{Q}^k_{\tube,s_2}(n_2)$,  in a tube $\tube=\tube_{N,M}$ is defined by an algorithm for $k=1,2$
to yield a unique element $\theta_3=\theta_1~(+)_{k}~\theta_2\in \mathcal{Q}^k_{\tube,s_1+s_2+d_k}(n_1+n_2+e_k)$.  Details of the algorithm for $k=2$ and $N+M\geq 3$ are given in \cite[Lemma 3.3]{Atapour10} with $d_2=18+8(N+M)$ and $e_2=4d_2+8(N+M)+4$, and for $k=1$ and $N+M\geq 2$ in \cite{BES19}.
The algorithm for $k=2$ is based on the following steps: 
$\theta_2$ is $x$-translated $s_1+d_k$ units so that it is $d_k$ units to the right of $\theta_1$. Then $k$ rightmost edges (the ``bottom-most" edge  from each polygon in the plane $x=s_1$) from $\theta_1$ are removed and $k$ leftmost edges (the ``bottom-most" edge  from each polygon in the plane $x=s_1+d_k$) from $\theta_2$  are removed. $e_k$ edges (whose union is topologically $2k$ embedded arcs) are added within the span-$d_k$ region between $\theta_1$ and $\theta_2$ to join the two $k$SAPs in such a way that a new $k$SAP in $\tube$ is formed. Figure \ref{fig:tangledconcat} shows a 2SAP $\theta_1$ (dashed lines) on the left and the two edges in the $x=s_1$ plane to be removed.  The block between the two planes has span $d_2/2$ and the red and blue lines indicate half the added edges ($e_2/2$) needed to concatenate $\theta_1$ to an arbitrary $\theta_2$ on the right.     For a $\theta_2$ a similar block of red and blue edges can be obtained to its left and the two blocks can be joined at the common pair of middle edges to create a single 2SAP $\theta_3$.   For $k=1$, one can focus just on the blue edges, however, less space is needed to match up both sides so that $d_1<d_2$ and $e_1<e_2$.  For sufficiently large $N,M$ it is expected that an analogous algorithm can be defined for $k>2$ but we do not provide details here.
\end{definition}

Note that the nature of the concatenation for $k=1$ gives that the knot-type of $\theta_3$ is given by the knot-type of $\theta_1\#\theta_2$, ie the concatenation operation is  a connected sum operation. 

For $k=2$, the operation involves a connected sum of components of $\theta_1$ and $\theta_2$ but it is not itself a connected sum operation.  However, as discussed in the next subsection, the operation is an example of a tangle product.

Using a concatenation operation that satisfies Definition \ref{def:concat}, we have, for example, for $d=d_k$ that 
\begin{equation}
\label{supermult1}
q^{(k)}_{\tube,s_1}q^{(k)}_{\tube,s_2} \leq q^{(k)}_{\tube,s_1+s_2+d}\leq 2^{A(N,M)(s_1+s_2+d)},
\end{equation}
where $A(N,M)$ is the number of lattice edges per span in an $(N,M)$-tube, 
and thus
\begin{equation}
  \label{supermult2}  
\log q^{(k)}_{\tube,s_1-d}+ \log q^{(k)}_{\tube,s_2-d} \leq \log q^{(k)}_{\tube,s_1+s_2-d}\leq A(N,M)(s_1+s_2+d)\log 2.
\end{equation}
Thus, the sequence $(\log q^{(k)}_{\tube,s-d}, s> d)$ is \emph{superadditive}, and by Fekete's lemma (see for example \cite[Lemma 1.2.1]{steele1997probability}), we have that the following limit exists
$$
\log \omega^{(k)}_\tube = \lim_{s\to \infty} \frac{1}{s} \log q^{(k)}_{\tube,s}  = \sup_{s} \frac{1}{s} \log q^{(k)}_{\tube,s-d},
$$
and thus the sequence $q^{(k)}_{\tube,s}$ grows exponentially in $s$. $\omega^{(k)}_\tube$ is referred to as the exponential growth rate for $q^{(k)}_{\tube,s}$ or, for short, the exponential growth rate for span-$s$ $k$SAPs. 
Similarly, if size is measured by length instead of span, the following limit exists for $e=e_k$:
$$
\log \Omega^{(k)}_\tube = \lim_{n\to \infty} \frac{1}{n} \log q^{(k)}_{\tube}(n)  = \sup_{n} \frac{1}{n} \log q^{(k)}_{\tube}(n-e).
$$
$\Omega^{(k)}_\tube$ is referred to as the exponential growth rate for  length-$n$ $k$SAPs. 

Regarding pattern theorems for $k$SAPs in tubes, proofs based on transfer matrix methods \cite{Sot98,atapourphdthesis,eng_phd_thesis} as well as proofs based on pattern-insertion strategies \cite{atapourphdthesis,Atapour10}  have been used for $k=1,2$.  Here we define patterns in a way that is consistent with both proof approaches. First, given any $\tube$, any span-$s$ $k$SAP $\pi$ in $\tube$, any non-negative $b\leq s+1$ and any non-negative $j\leq s+1$,  we refer to the set of all edges and half-edges of $\pi$ that occur between the planes $x=j-1/2$ and $x=b+j-1/2$ as the span-$b$ \textit{block} of $\pi$ at $x=j-1/2$.
A \textit{$k$SAP pattern} $P$ is then defines as an ordered pair $(\eta_P,E_P)$ where: there exists a $k$SAP $\pi$ such that for some choice of $b$ and $j$, $E_P$ occurs as a block of $\pi$ and such that $\eta_P$ corresponds to the \textit{left-connect} information of $E_P$ in $\pi$, that is $\eta_P$ is a set of $k$ pair-partitions such that the pairs in the $i$th set indicate which left-most end-points (points of degree-one in the half-integer plane) in $E_P$ are connected to each other on the left in the $i$th component of the $k$SAP $\pi$.
The span of a $k$SAP pattern is defined to be $b$, that is the $x$-span of the block that satisfies the definition. 
In the case $\pi$ has span $s$ and the choice of $b$ and $j$ are such that $j=0$ ($j>0$) and $b\leq s$ ($b=s+1-j$) then the pattern is called a \textit{start pattern} (\textit{end pattern}). While if $j>0$ and $b\leq s-j$ so that $E_P$ occurs as a block that does not intersect the start or end $x$-plane of $\pi$, then the pattern is called a \textit{proper pattern}. 
We refer to a proper $k$SAP pattern $(\eta_P,E_P)$ as a \textit{link pattern}  if  $E_P$ has exactly $2k$ half-edges intersecting each of its left-most and right-most $x$-planes and if the pairs of end-points that must be connected to each other on the left and on the right, as dictated by $\eta_P$,  are unit distance apart from each other.  Note that a  link pattern can be unambiguously  closed off into a span-$(b+1)$ $k$SAP by joining, as dictated by $\eta_P$,  each end-point pair on the left and right using exactly two half-edges and one edge. We thus define the link-type of a link pattern by the link-type of the resulting span-$(b+1)$ $k$SAP.

For $k=1,2$, the  concatenation operation combined with the fact that one can define a finite transfer matrix to generate $k$SAPs, leads to the following {\it pattern theorem}. We state it here for the cases that $k$SAP size is measured by span or by length but analogous results hold for more general settings (see \cite{atapourphdthesis,Madras99}).

\begin{thm}[\cite{Sot98,Atapour10}]\label{patternthm}
    For $k=1,2$, given a proper $k$SAP pattern $P$, let $q^{(k)}_{\tube,s}(P, < \ell)$ denote the number of span-$s$ $k$SAPs that contain fewer than $\ell$ $x$-translates of $P$.  Then there exists $\epsilon_P >0$ such that
    $$ \limsup_{s\to\infty} \frac{1}{s} \log q^{(k)}_{\tube,s}(P, < \epsilon_Ps) < \log \omega^{(k)}_\tube .$$
    Similarly, let $q^{(k)}_{\tube}(n;P, < \ell)$ denote the number of length-$n$ $k$SAPs that contain fewer than $\ell$ $x$-translates of $P$.  Then there exists $\tilde{\epsilon}_P>0$ such that
    $$ \limsup_{n\to\infty} \frac{1}{n} \log q^{(k)}_{\tube}(n;P, < \tilde{\epsilon}_Pn) < \log \Omega^{(k)}_\tube .$$
\end{thm}
Thus, as $s\to\infty$, all but exponentially few span-$s$ $k$SAPs contain the pattern $P$ at least $\epsilon_P s$ times. 
Similarly, as $n\to\infty$, all but exponentially few length-$n$ $k$SAPs contain the pattern $P$ at least $\tilde{\epsilon}_P n$ times.

For the set of knot-types ${\cal K}={\cal L}^1$, Soteros, Sumners and Whittington \cite{SSW92} defined a  good measure of knot complexity to be a function $F:{\cal K}\to [0, \infty)$
which satisfies the following:
(i) $F(unknot) = 0$; (ii) there exists $K\in {\cal K}$ such that $F(\#_{i=1}^{n}K \# L) \geq  nF(K) > 0$ for all $L\in {\cal K}$.
Then, since concatenation for SAPs is a connected sum operation,  a known pattern theorem for unconfined SAPs, allowed them to show that all but exponentially few sufficiently long length-$n$ SAPs would have $F$ complexity $\geq \beta n$ for some $\beta >0$ (see \cite[Corollary 3.2]{SSW92}).

Here, we use tangle products to generalize the concept of a good measure of knot complexity to a good measure of entanglement complexity that is applicable to $k$SAPs. The measures will be defined on  the set of smooth ambient isotopy classes of $k$-component links, denoted by $\mathcal{L}^k$. Instead of connected sums, 
 tangle products will be used. Tangle products are defined in more detail in the next subsection.

\subsection{Tangle Products}\label{Sec:TangleProd}

Concatenation in $k$SAPs is much harder to control than concatenation for polygons. Figure \ref{fig:tangledconcat}  demonstrates that a concatenation of two unlinks does not necessarily produce an unlink. This is in contrast to the connected sum, where the connected sum of two unknots always produces the unknot.

Just as the concatenation of 1SAPs corresponds to the topological operation of connected sum of knots, the concatenation of $k$SAPs corresponds to the operation of tangle product. Next, we develop the definition of tangle product.

A \textit{tangle} is an ordered pair $(B,t)$ where $B$ is a 3-ball and $t\subset B$ is a properly embedded collection of arcs and loops. A \textit{rational tangle} is a tangle $(B,t)$ such that $t$ is a collection of boundary parallel arcs \cite{cromwell2004knots}.

\begin{definition}  
 A graph $G$ is an $n$-star graph if $G$ has $n$ edges and $n+1$ vertices such that $n$ of the vertices are valence one and one of the vertices is valence $n$. Denote by $\partial(G)$ the set of valence one vertices.
\end{definition}

\begin{definition}
\label{def:product}
Let $K_1$ and $K_2$ be links embedded in distinct copies of $S^3$, denoted $S^3_1$ and $S^3_2$ respectively. Let $G_1$ and $G_2$ be $n$-star graphs embedded in $S^3_1$ and $S^3_2$ respectively such that $G_i\cap K_i=\partial(G_i)$. Let $\mu(G_i)$ be a small, closed, regular neighborhood of $G_i$ in $S^3_i$ such that $(\mu(G_i), K_i \cap \mu(G_i))$ is a rational tangle. Let $B_{i}=S^3_i-\text{int}(\mu(G_i))$. A link in $S^3$ obtained by gluing $\partial(B_1)$ to $\partial(B_2)$ via a homeomorphism such that points in $\partial(B_1)\cap K_1$ are mapped to points in $\partial(B_2)\cap K_2$ is called an \textit{$n$-strand tangle product} of $K_1$ and $K_2$ and is denoted by $K_1\ast_{S}K_2$. The image of $\partial(B_1)$ and $\partial(B_2)$ under this identification is called the product sphere and is denoted $S$. (See Figure \ref{fig:figures2} (a) for an example of a 2-strand tangle product.) 
\end{definition}

In particular, connected sums are 1-strand tangle products. In the literature, 2-strand tangle products are also known as \textit{Generalized Conway products}. The bridge number of both Generalized Conway products and tangle products has been previously studied \cite{blair2010bridge,blair2013bridge}. Note that the topological type of a tangle product $K_1\ast_{S}K_2$ is dependent on the topological type of $K_1$, the topological type of $K_2$, the isotopy type of $G_1$ in $(S^3_1, K_1)$, the isotopy type of $G_2$ in $(S^3_2, K_2)$ and the isotopy type of the homeomorphism that maps $\partial(B_1)$ to $\partial(B_2)$. This makes understanding the behavior of topological properties under tangle product much more challenging than under connected sum.

Concatenation of two $k$SAPs in a tube $\tube$ corresponds to a tangle product. To see this, take a $k$SAP $\theta_1$ in a tube $\tube_1$ that we think of as being embedded in a copy of the 3-sphere, $S^3_1$. Similarly, take a $k$SAP $\theta_2$ in a tube $\tube_2\subset S^3_2$. Let $G_1$ be a $k$-star graph embedded to the right of $\theta_1$ in $S^3_1$ that consists of a center vertex and $k$ monotone (with respect to projection onto the $x$-direction) arcs that travel from the center vertex to each of the rightmost edges from $\theta_1$. Similarly, let $G_2$ be a $k$-star graph embedded to the left of $\theta_2$ in $S^3_2$ that consists of a center vertex and $k$ monotone arcs that travel from the center vertex to each of the leftmost edges from $\theta_2$. Then every concatenation of  $\theta_1$ and $\theta_2$ is topologically a tangle product $K_1\ast_{S}K_2$ that utilizes star graphs $G_1$ and $G_2$. Note that the structure of the concatenation determines the isotopy types of $G_1$ and $G_2$, but allows for any isotopy type of the homeomorphism that maps $\partial(B_1)$ to $\partial(B_2)$.

\subsection{Good measures of spanning link complexity in a general setting}\label{sec:linkGoodMeasures}

For good measures of knot complexity, a sequence of connected sums $nK\# L= \#_{i=1}^{n}K \# L$  was considered. In this case, the resulting knot-type does not depend on the order of the connected sums. For the good measures of link complexity introduced here, we need to consider sequences of tangle products.  The sequence of operations is motivated by considering how a sequence of $k$SAPs could be concatenated together.  

For this, a fixed $k$SAP $\theta$ can be concatenated on the right or the left with another $k$SAP. However, one must specify the $k$ rightmost edges from $\theta$ (one from each polygon) and the $k$ leftmost edges from $\theta$ (one from each polygon), respectively, to achieve these concatenations. From the topological perspective, this is equivalent to the link type of $\theta$ coming equipped with the additional information of two embedded $k$-star graphs which define concatenation on the right and concatenation on the left. Thus to obtain an associated sequence of tangle products, we introduce the following definition of \textit{k-tangle-triple}.

\begin{definition}
\label{def:tangle}
    A \textit{k-tangle-triple} is a triple $(L,G_1,G_2)$ where $L$ is a $k$-component link embedded in $S^3$, each $G_i$ is a $k$-star graph embedded in $S^3$ such that $G_1\cap G_2=\emptyset$, $G_i\cap K_i=\partial(G_i)$ for each $G_i$, and each $G_i$ intersects each component of $K$ in exactly one point. Two $k$-tangle-triples $(L^1,G^1_1,G^1_2)$ and $(L^2,G^2_1,G^2_2)$ are equivalent if there is a smooth ambient isotopy of $S^3$ that takes $L^1$ to $L^2$, $G^1_1$ to $G^2_1$ and $G^1_2$ to $G^2_2$. The set of equivalence classes of $k$-tangle-triples is denoted by $\mathcal{L}_T^k$.
\end{definition}

Given two $k$-tangle-triples $(L^1,G^1_1,G^1_2)$ and $(L^2,G^2_1,G^2_2)$ we can define $(L^1,G^1_1,G^1_2)\underline{*}(L^2,G^2_1,G^2_2)$ to be any $k$-tangle-triple given by $(L^1\ast_{S}L^2,G^1_1,G^2_2)$ where $L^1\ast_{S}L^2$ is any tangle product formed by taking $G^1_2$ as the $k$-star graph for $L^1$ and $G^2_1$ as the $k$-star graph for $L^2$. As discussed in Section \ref{Sec:TangleProd},  the link type of $L^1\ast_{S}L^2$ is dependent on the isotopy class of the attaching homeomorphism from the boundary of a regular neighborhood of $G^1_2$ to the boundary of a regular neighborhood of $G^2_1$. Note that we restrict which attaching homeomorphisms we allow, so that $(L^1\ast_{S}L^2,G^1_1,G^2_2)$ continues to meet the definition of $k$-tangle-triple (i.e. $L^1\ast_{S}L^2$ is still a $k$-component link). This restriction includes all homeomorphisms that result in a component-wise connected sum of $L^1$ and $L^2$. Hence, $(L^1\ast_{S}L^2,G^1_1,G^2_2)$ may correspond to infinitely many distinct $k$-tangle-triples, depending on the attaching homeomorphism. In the statements that follow, we abuse notation slightly and let $(L^1\ast_{S}L^2,G^1_1,G^2_2)$ denote some fixed $k$-tangle-triple constructed from some arbitrary attaching homeomorphism. Additionally, we will sometimes suppress the additional notation and refer to $(L^1\ast_{S}L^2,G^1_1,G^2_2)$ as $L^1\underline{*}L^2$. This allows us to define a good notion of link complexity. When comparing the operation $\underline{*}$ on $k$-tangle-triples to the operation of concatenation on $k$SAPs, $L^1$ is playing the role of $\theta_1$, $L^2$ is $\theta_2$, $G^1_2$ is the ``right'' of $\theta_1$, $G^2_1$ is the ``left'' of $\theta_2$, the $k$ arcs of $L^1$ in a neighborhood of $G^1_2$ are the $k$ rightmost edges removed from $\theta_1$, the $k$ arcs of $L^2$ in a neighborhood of $G^2_1$ are the $k$ leftmost edges removed from $\theta_2$, and  $L^1\ast_{S}L^2$ is the topological type of the $k$SAP that results from concatenating $\theta_1$ and $\theta_2$.

Recall that $\mathcal{L}^k$ is the set of smooth ambient isotopy classes of $k$-component links. Given any link complexity function $F:\mathcal{L}^k\rightarrow [0,\infty)$ we can expand $F$ to a function $F:\mathcal{L}_T^k\rightarrow [0,\infty)$ by declaring 
$F((L,G_1,G_2))=F(L)$. Given a knot $K\in \mathcal{L}^1$, we define $(L(K),G_1,G_2)\in \mathcal{L}_T^k$ as in Figures \ref{fig:figures3} (a) and \ref{fig:L(K)}. Notice that the link $L(K)$ is the result of taking a $k$-component unlink and connected summing a copy of $K$ to a single component.

\begin{figure}[h]
\begin{subfigure}{0.45\textwidth}
    \labellist
\small
\pinlabel{$G_1^1$} at -19 92
\pinlabel{$G_2^1$} at 290 49
\pinlabel{$K^1$} at 153 192
\pinlabel{$G_1^2$} at 353 49
\pinlabel{$G_2^2$} at 670 92
\pinlabel{$K^2$} at 500 192
\pinlabel{$K^1$} at 947 192
\pinlabel{$K^2$} at 1290 192
\endlabellist\includegraphics[width=\textwidth]{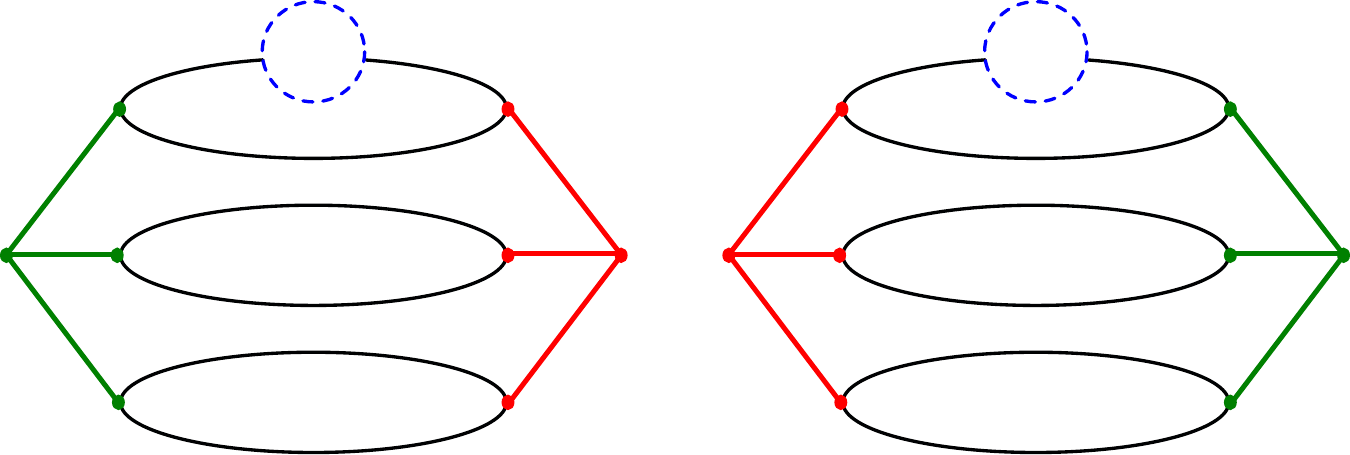}
    \caption{}
    \label{fig:withstar}
\end{subfigure}
\hfill
\begin{subfigure}{0.45\textwidth}
    \includegraphics[width=\textwidth]{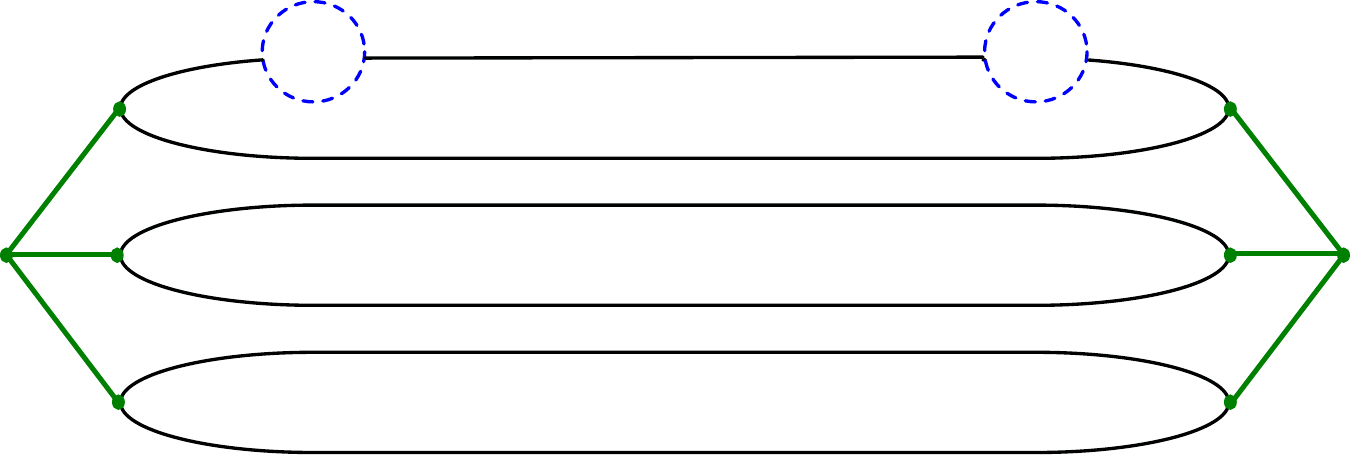}
    \caption{}
    \label{fig:concatstar}
\end{subfigure}  
\caption{(a) Two $3$-tangle-triples $(L(K^1),G_1^1,G_2^1)$ and $(L(K^2),G_1^2,G_2^2)$ where $K^i$ for $i=1,2$ is any knot. (b) A result of performing a 3-strand tangle product to get $(L(K^1\#K^2),G_1^1,G_2^2)$.}
\label{fig:L(K)}
\end{figure}
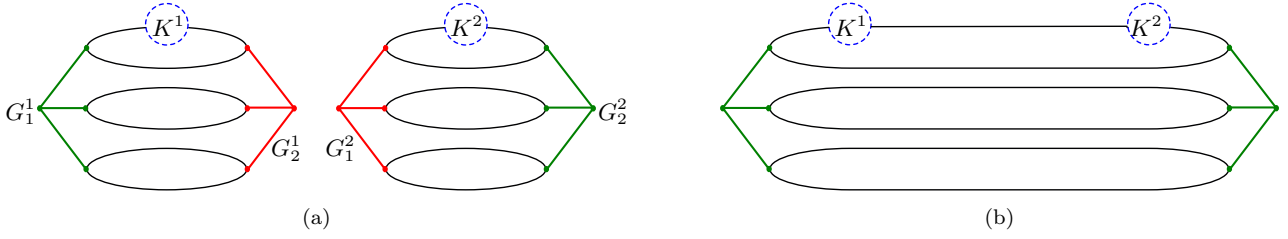

\begin{definition}
 A \textit{good measure of $k$-component spanning link complexity} is a function $F:\mathcal{L}^k\rightarrow [0,\infty)$ which satisfies the following:
(i) $F(k$-component unlink) = 0; (ii) there exists $L \in \mathcal{L}_T^k$ such that for any integer $m$ and any $J_1, J_2,...,J_{m+1}\in \mathcal{L}_T^k$, 
$F(J_1\underline{*} L \underline{*} J_2 \underline{*} L \underline{*} J_3 \underline{*}...\underline{*}J_{m}\underline{*}L\underline{*}J_{m+1}) 
\geq \frac{m}{k}F(L) > 0$. \label{def:goodmeasure}

\end{definition}

Note that by this definition a good measure of $1$-component spanning link complexity  coincides with the definition of a good measure of knot complexity given in \cite{SSW92}. Hence, we will refer to a good measure of $1$-component spanning link complexity as a good measure of knot complexity.

\LinkMeasureToGoodMeasure*
\begin{proof}
    By hypothesis 3, $F:\mathcal{L}^k\rightarrow [0,\infty)$ is a good measure of $k$-component spanning link complexity when $k=1$.
    
    Assume $k\geq 2$. Since $F:\mathcal{L}^1\rightarrow [0,\infty)$ is a good measure of knot complexity. There exists $K \in \mathcal{L}^1$ such that $F(\#_{i=1}^{m}K\#J) \geq mF(K) > 0$ for all $J\in \mathcal{L}^1$. Suppose that $J_1, J_2,...,J_{m+1}\in \mathcal{L}_T^k$ and $M\in \mathcal{L}_T^k$ has a decomposition of the form $J_1\underline{*} L(K) \underline{*} J_2 \underline{*} L(K) \underline{*} J_3 \underline{*}...\underline{*}J_{m}\underline{*}L(K)\underline{*}J_{m+1}$. Then there exists $M'\in \mathcal{L}_T^k$ of the form $J_1 \underline{*} J_2 \underline{*} J_3 \underline{*}...\underline{*}J_{m}\underline{*}J_{m+1}$ such that each copy of $L(K)$ in the decomposition of $M$ corresponds to connected summing a single copy of $K$ to some component of $M'$ to form $M$. In particular, some component $C_i$ of $M$ has a connected sum decomposition containing at least $\lceil \frac{m}{k} \rceil$ copies of $K$. By hypothesis 4, $F(M)\geq F(C_i)$. By definition of good measure of knot complexity, $F(C_i)\geq \lceil \frac{m}{k} \rceil F(K)\geq \frac{m}{k} F(K)>0$. By hypothesis 2, $F(K)=F(L(K))$. Hence, $F(M)\geq \frac{m}{k} F(L(K))>0$.  Since $M$ was arbitrary, $F(J_1\underline{*} L \underline{*} J_2 \underline{*} L \underline{*} J_3 \underline{*}...\underline{*}J_{m}\underline{*}L\underline{*}J_{m+1}) 
\geq \frac{m}{k}F(L)>0.$

\end{proof}

From the arguments in \cite[Theorem 3.3]{SSW92}, the following knot invariants give a good measure of knot complexity: number of prime
factors, genus, bridge number minus one, span of any non-trivial Laurent knot
polynomial, log(order), crossing number, unknotting number, minor index, braid
index minus one (for definitions, see \cite{SSW92}).   

At least four of these knot invariants can be modified slightly to satisfy Theorem \ref{thm:LinkMeasureToGoodMeasure}. We now list these invariants, together with brief reasoning. In the following list, condition 3 of Theorem \ref{thm:LinkMeasureToGoodMeasure} holds due to \cite{SSW92}.
\begin{enumerate}
    \item \textbf{Crossing number}: We observe that the unlink admits a diagram with no crossings. Hence, condition 1 of Theorem \ref{thm:LinkMeasureToGoodMeasure} is satisfied. Next, if the link is the split union of an unlink with another link $L'$, then we can produce an isotopy after which there is a projection of the link to a crossing-less diagram for the unlink part. Thus, condition 2 holds as the crossing number contribution will come from the $L'$ components. The crossing number satisfies condition 4 because a diagram of any component knot appears in a diagram of the link. Therefore, the minimum number of crossings needed to represent the entire link cannot be less than the crossing number of the component knot. 
    \item \textbf{Bridge number minus the number of components} or  \textbf{Braid index minus the number of components}: It is well-known that both the bridge number and the braid index of a $k$-component unlink is exactly $k$. Subtracting $k$ will give condition 1. To see that condition 2 holds, suppose that $L'$ has $r$ components and $U$ has $s$ components. Then, for example, $F(L')=\beta(L')-r$, which is the same as $F(L)= \beta(L)-(r+s)=\beta(L')+s-(r+s)$. Here, the last equality is due to the fact that the unlink contributes $s$ to the bridge number. A similar argument holds for braid index. To see that condition 4 holds for $F$ being either one of these invariants, let $K$ be any knot and consider any link $L$ with $k$ components, $K$, $L_1$, ... ,$L_{k-1}$. Observe that $F(K\cup L_1 \cup L_2 \cup \cdots \cup L_{k-1})\geq F(K)+F(L_1)+\cdots + F(L_{k-1}) \geq F(K).$
      \item \textbf{Unlinking number}: Recall that the {unlinking number} is the minimum number of crossing changes over all diagrams of $L$ that turns $L$ into the unlink. It is obvious why condition 1 and 2 are met. For condition 4, to unlink $L$, one has to turn $K$ into the unknot as well. Thus for $F$ corresponding to unlinking number,  we need to perform at least $F(K)$ crossing changes to unlink $L$ and hence $F(L)\geq F(K)$. 
\end{enumerate}
 The examples below show some invariants that do not satisfy Theorem \ref{thm:LinkMeasureToGoodMeasure}, but will be shown to be good measures via different methods as part of the proof of Theorem \ref{thm:moregoodmeasures}  in Section \ref{subsec:additional}.
\begin{enumerate}
    \item \textbf{Genus}: {Recall that the {genus} of a link is the the minimum genus over all connected orientable surfaces bounded by the link. To see that taking $F$ to be genus does not satisfy condition 4 of Theorem \ref{thm:LinkMeasureToGoodMeasure}, consider taking a link formed from an index-2 cable of a knot $K$ with high genus. This parallel link bounds an annulus, implying that $F(L) = 0$. On the other hand, $K$ is still a component of $L$ and $F(K)$ can be arbitrarily large. This argument also shows that the \textbf{Minor index} does not satisfy condition 4 as the minor index is related to the size of the Seifert matrix, which is obtained from curves on the surface bounded by the link. So, if the genus of the link is lower than the genus of a component knot, then the minor index of $L$ can be less than that of $K$.}
    \item \textbf{Number of prime factors}: {Taking $F$ to be the number of prime factors does not satisfy condition 4 of Theorem \ref{thm:LinkMeasureToGoodMeasure}. An example to demonstrate this  can be seen in Figure \ref{fig:exterior}. Condition 4 says that for any knot $K$ and for any link $L$ with component knot $K$, $F(L)\geq F(K)$. In the case depicted in Figure \ref{fig:exterior}, $F(L) = 0$, but $F(K)$ is $1$.}
    \item \textbf{Number of $p$-colorings}: (See subsection \ref{subsec:additional} for more details.) {The number of $p$-coloring does not satisfy properties for Theorem \ref{thm:LinkMeasureToGoodMeasure}. The example in Figure \ref{fig:pcoloringfails} shows why this complexity violates condition 4. Condition 4 says that for any knot $K$ and for any link $L$ with component knot $K$, $F(L)\geq F(K)$. For each link $L$ depicted in Figure \ref{fig:pcoloringfails}, $F(L) = 9$, but $F(K)$ is $3^{m+1},$ where $m$ is the number of trefoil summands in the component $K$.}
    \item \textbf{Span of the Alexander polynomial and $\log$(order)}: Depending on the link polynomial invariant, the span of the polynomial may not satisfy properties in Theorem \ref{thm:LinkMeasureToGoodMeasure}. For instance, suppose that $L$ is a split union where a component knot $K$ has large Alexander polynomial span. Since the Alexander polynomial of a split link is 0, it follows that  $F(L)< F(K)$. Since the $\log$(order) is determined by the Alexander polynomial. This same example shows that $\log$(order) fails condition 4.
\end{enumerate}

\begin{figure}[h]
\centering
\begin{subfigure}{0.3\textwidth}
 \includegraphics[width=\textwidth]{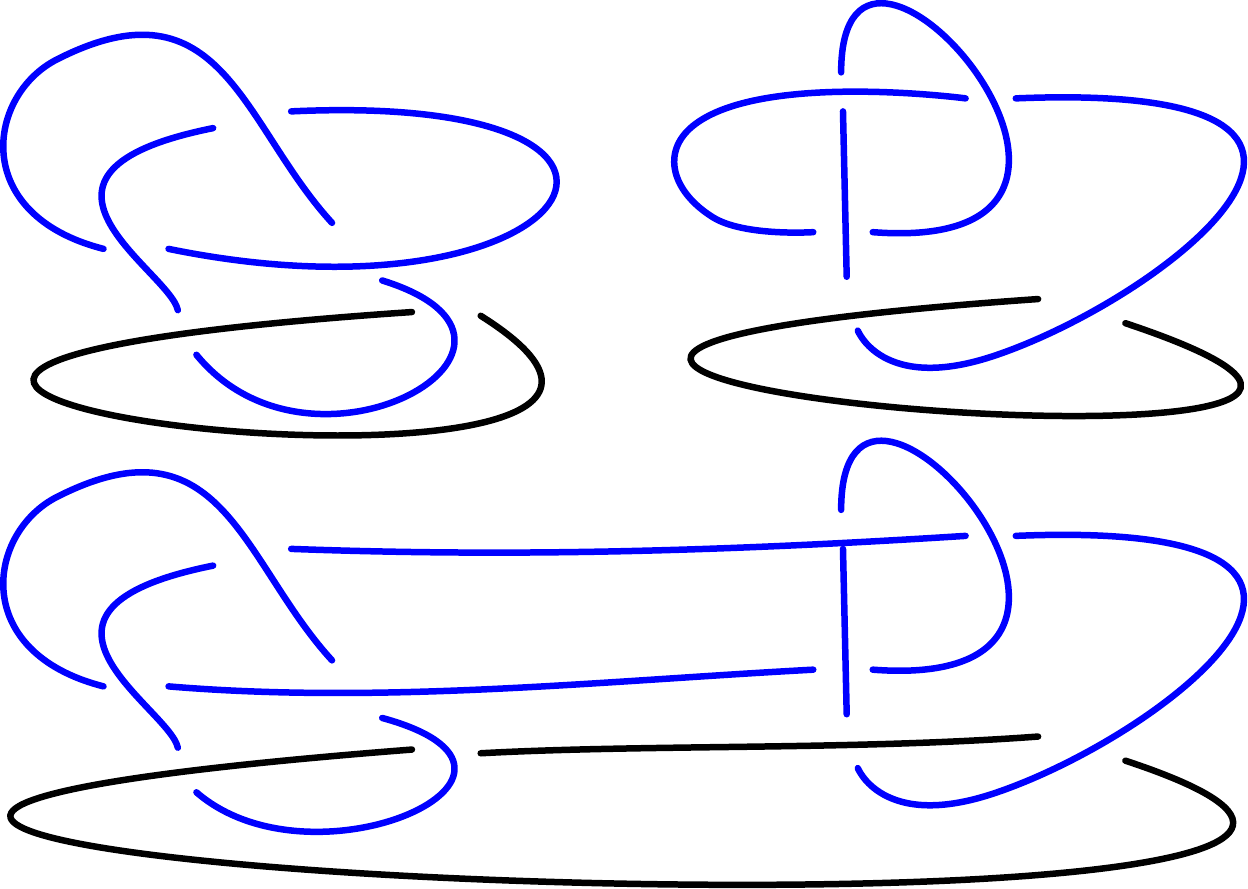}
    \caption{}
    \label{fig:exterior}
\end{subfigure}
\hspace{.5cm}
\begin{subfigure}{0.3\textwidth}
    \includegraphics[width=\textwidth]{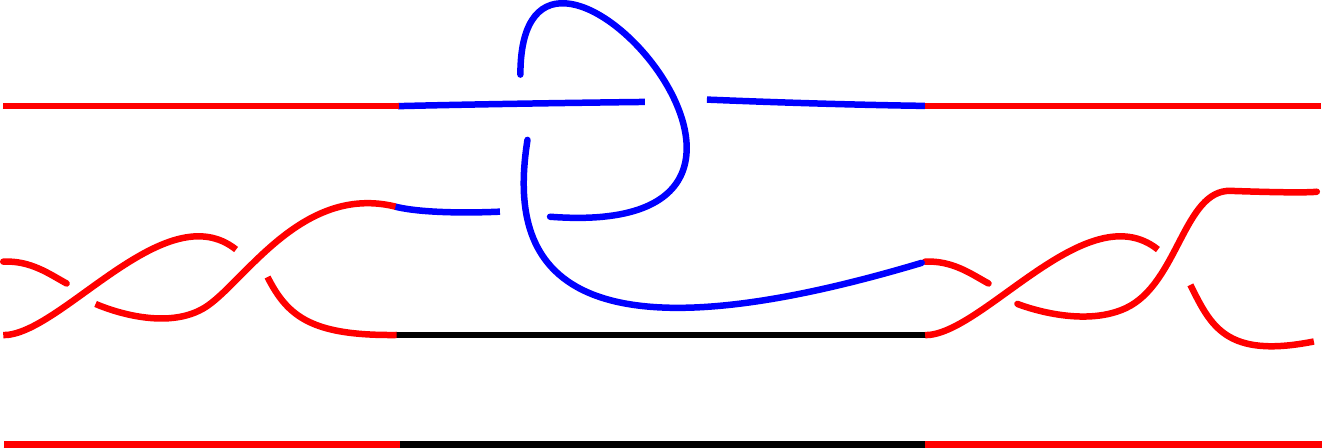}
    \caption{}
    \label{fig:nonlocal}
\end{subfigure}  
\caption{(a) The number of prime knot factors may decrease under concatenation. The resulting link after concatenation is the prime link $L10a38$ from Thistlethwaite's table of prime links. (b) A closure of this pattern is a prime link $L7a2.$}
\label{fig:exteriorandnonlocal}
\end{figure}

\begin{figure}[ht!]
\centering
\includegraphics[width=3cm]{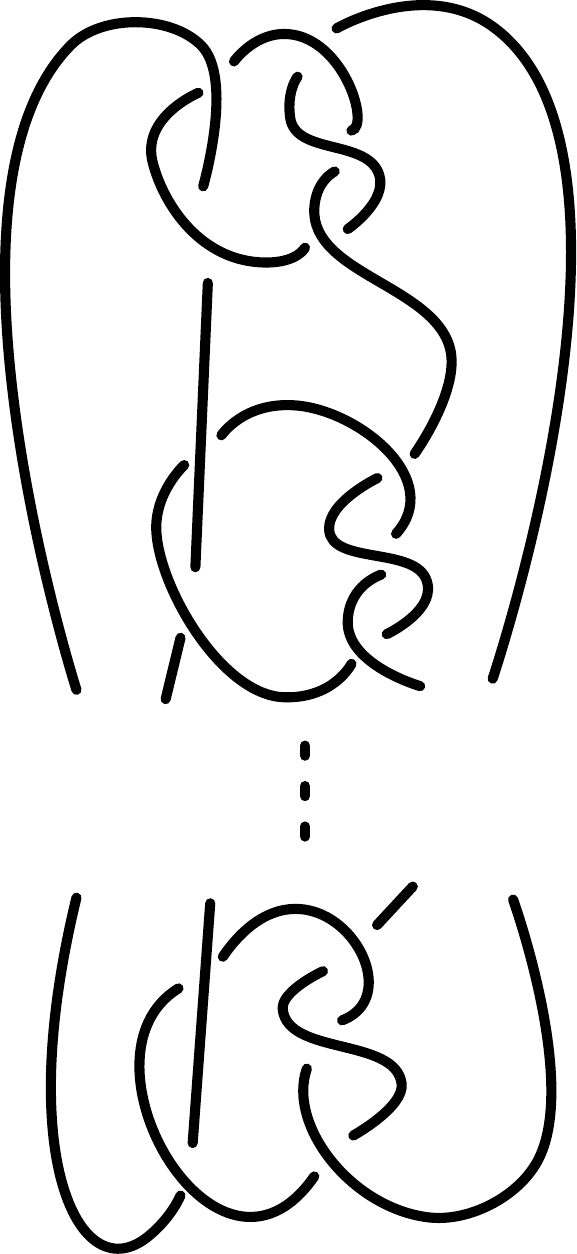}
\caption{A family of 2-component links that shows that the number of $p$-colorings fails condition 4 of Theorem \ref{thm:LinkMeasureToGoodMeasure}. Each link $L$ in this family has one component the unknot and the other, $K$, is a connected sum of $m\geq 1$ trefoils. Such an $L$ can be shown to have $9$ distinct 3-colorings. Additionally, $K$ has  $3^{m+1}$ distinct 3-colorings.}
\label{fig:pcoloringfails}
\end{figure}

\goodmeasure*
\begin{proof}
We will verify conditions 1-4 of Theorem \ref{thm:LinkMeasureToGoodMeasure}. Suppose $F_1:\mathcal{L}^1\rightarrow [0,\infty)$ is a good measure of knot complexity. If $U$ is an unlink, then $F(U)=max_{K\subset U} F_1(K)=F_1($unknot$)=0$, satisfying condition 1. To see that condition 2 holds, suppose $L=L'\cup U$ is a split link with unlink $U$. Then, $F(L)=\max(F(L'),F(U)) = \max(F(L'),0)=F(L')$. Here, the second equality holds because we showed that $F$ satisfies condition 1. For condition 3, $F$ restricted to $\mathcal{L}^1$ is $F_1$ which is given to be a good measure of knot complexity. Finally, condition 4 also holds because $F(L)=max_{K\subset L} F_1(K)\geq F_1(K_0)=F(K_0)$ for every component $K_0$ of $L$.

\end{proof}

By applying the arguments in \cite[Theorem 3.3]{SSW92} and Corollary \ref{cor:thm1}, taking the component-wise maximum of each of the following knot invariants gives a good measure of $k$-component spanning link complexity: number of prime
factors, genus, bridge number minus one, span of any non-trivial Laurent knot
polynomial, log(order), crossing number, unknotting number, minor index, braid
index minus one.

In the next subsection we prove Theorem \ref{thm:moregoodmeasures} and show, for example, that the number of $p$-colorings is a good measure.  We start by reviewing  necessary definitions.

\subsection{Additional good measures of $k$-component spanning link complexity} \label{subsec:additional}

Arguably, the most basic and computable invariant to distinguish links is tricolorability \cite{adams2004knot}, which is a special case of Fox $p$-colorings \cite{fox1962quick}. For this, one assigns a color to each strand in a link diagram with up to  three distinct colors, obeying some given rules. By a \textit{strand} in a diagram of a link, we mean a connected component of the link diagram that goes from one undercrossing to another with only overcrossings in
between. We will say that a diagram of a link is \textit{tricolorable} if
each of the strands can be colored using one of the three different
colors, so that at each crossing, either three different colors come together
or all the same color comes together. For a link diagram to be tricolorable, we further require that at least two distinct colors are used. In more generality, a diagram of a link is Fox $p$-colorable (or just \emph{$p$-colorable}) if each strand can be given a label from $\{0,1,\cdots, p-1\}$ such that the relation $z=2y-x \mod p$ holds at every crossing, where $y$ is the overstrand and $x,z$ are the two understrands. We again require that at least two distinct labels appear in a diagram. Having a $p$-colorable link diagram is a link invariant for $p\geq 3$. Moreover the number of $p$-colorings of a link diagram is a link invariant, denoted $Col_p(L)$, where monochromatic colorings are included in this count.

For example, the unlink on at least 2 components and the link $L6a1$ are tricolorable (see Figure \ref{fig:tricol}), while the Hopf link and the Whitehead link are not (see arguments in \cite{adams2004knot}). Additionally,  $Col_p(L)=p^k$ when $L$ is a $k$-component unlink.

\begin{figure}[h]
\centering
\includegraphics[width=3cm]{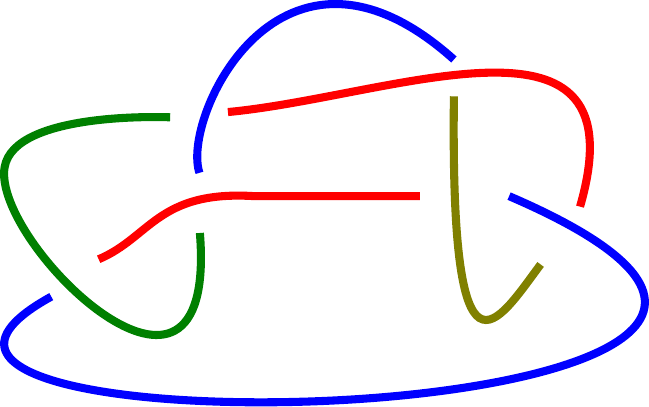}
\caption{The link $L6a1$ is tricolorable}
\label{fig:tricol}
\end{figure}

This coloring idea can be generalized from using integers modulo $p$ to using elements of any group. More precisely, one can show that a link is tricolorable if there exists a surjective homomorphism from the fundamental group of the link exterior $\pi L$ to the 6 element dihedral group. Similarly, we can obtain $Col_3(L)$ by counting the number of homomorphisms from $\pi L$ to the 6 element dihedral group where meridians of the link are sent to reflections (see Example 1.1 of \cite{eisermann2007knot}). More generally, one can define a link invariant by counting the number of homomorphisms from $\pi L$ to any group.

Using a different argument than that in Theorem \ref{thm:LinkMeasureToGoodMeasure}, we prove the following result that applies to another class of link invariants including $p$-colorability.

\moregoodmeasures*

 \begin{proof}
 We only need to check condition (ii) in the definition of a good measure. That is, we show that there exists $L \in \mathcal{L}_T^k$ such that for any integer $m$ and any $J_1, J_2,...,J_{m+1}\in \mathcal{L}_T^k$, $F(J_1\underline{*} L \underline{*} J_2 \underline{*} L \underline{*} J_3 \underline{*}...\underline{*}J_{m}\underline{*}L\underline{*}J_{m+1}) 
\geq \frac{m}{k}F(L) > 0$.

Take $K$ to be a knot that satisfies $\emph{(2)}$ in the statement of the theorem. Let $L$ be $(L(K),G_1,G_2)\in \mathcal{L}_T^k$ as in Figure \ref{fig:L(K)separated}. Let $m$ be a positive integer and $J_1, J_2,...,J_{m+1}\in \mathcal{L}_T^k$. Suppose $M\in \mathcal{L}_T^k$ has a decomposition of the form $J_1\underline{*} L(K) \underline{*} J_2 \underline{*} L(K) \underline{*} J_3 \underline{*}...\underline{*}J_{m}\underline{*}L(K)\underline{*}J_{m+1}$. Then the number of knot-type $K$ factors of $M$ is at least $m.$ By the assumption $\emph{(2)}$, we have that $F(M)\geq  mF(L)\geq \frac{m}{k}(F(L))$. Since $M$ was arbitrary, $F(J_1\underline{*} L \underline{*} J_2 \underline{*} L \underline{*} J_3 \underline{*}...\underline{*}J_{m}\underline{*}L\underline{*}J_{m+1}) 
\geq \frac{m}{k}F(L).$

\end{proof}

To demonstrate Theorem \ref{thm:moregoodmeasures}, note that the genus, the number of prime factors and the span of the Alexander polynomial are all zero for the unknot. Additionally, all three of these invariants are additive with respect to connected sum. Hence, Theorem \ref{thm:moregoodmeasures} implies they are good measures of $k$-component spanning link complexity when we choose $K$ to be the trefoil, for instance. Note that in the cases of the genus and the number of prime factors any non-trivial knot $K$ will do.

In the case of $p$-colorings, we fix $k\geq1$ take $F$ to be $\log_p(Col_p(L))$ minus the number of components of $L$, where $Col_p(L)$ is the number of Fox $p$-colorings of $L.$ Recall that $Col_p(L)=p^k$ when $L$ is a $k$-component unlink. Hence, $\log_p(Col_p(L))-k=0$ when $L$ is a $k$-component unlink. Take $K$ to be a $T(2,p)$ torus knot. It is well known that $Col_p(T(2,p))=p^2$. Thus, $Col_p(L(K))$ is the number of Fox $p$-colorings of the split union of a $(k-1)$-component unlink with the $T(2,p)$ torus knot, which is $p^{(k+1)}$. Taking $\log$ and subtracting by the number of components gives condition $\emph{(2)}$ of Theorem \ref{thm:moregoodmeasures}.

It is important to note that both the proof of Theorem \ref{thm:LinkMeasureToGoodMeasure} and the proof of Theorem \ref{thm:moregoodmeasures} rely on producing a particular type of $L \in \mathcal{L}_T^k$ to meet criteria (ii) in the definition of a good measure of $k$-component spanning link complexity. Specifically, each proof relies on choosing $L=L(K)$ for some suitably chosen knot $K$. Consequently, the link complexities that Theorem \ref{thm:LinkMeasureToGoodMeasure} and Theorem \ref{thm:moregoodmeasures} can show are good measures will only be complexities that increase linearly as the complexity of the prime decompositions of each component increases. However, Definition \ref{def:goodmeasure} can produce good measures of $k$-component spanning link complexity even when the $L$ in criteria (ii) is not of the form $L(K)$ for any $K$. For example, we can take $L$ to be one of the 2-tangle-triples which gives rise to the tangle in Figure \ref{fig:Incomp4}. We remind the readers of the following widely-studied invariant.

\begin{definition}
    The \textit{splitting number} of a link is the minimum number of crossing
changes between distinct components that is required to convert the link into
a split link.
\end{definition}

In what follows, we will enhance the tangle used in \cite{Atapour10} to guarantee that a link is not split no matter how the endpoints of the tangle connect to the rest of the link. We also generalize Theorem 6.3 of \cite{Atapour10} so that any inseparable tangle may be used. We will make use of a criterion originally due to Nakanishi \cite{nakanishi1981primeness,nakanishi1983prime}. Recall that a tangle $(C,v)$ is inseparable if there is no properly embedded disk in $C$ that is disjoint from $v$ and separates components of $v$. 

\begin{lemma}[\cite{nakanishi1981primeness,nakanishi1983prime}]
    \begin{enumerate}
\item[(1)]Let $(C,v)$ be a $2$-string tangle and let $D$ be a disk properly embedded in $C$ that intersects both arcs of $(C,v)$ in a single point each and separates $(C,v)$ into two $2$-string tangles $(A,t)$ and $(B,u)$. Let $(B,u)$ be an inseparable tangle. Suppose that for any disk $D'$ properly embedded in $A$ with
$D' \cap \partial D = \emptyset$ and
$D' \cap t = \emptyset
$, it holds that $D'$ does not separate $t$, then $(C,v)$ is an inseparable tangle.

\item[(2)] Let $L$ be a link in $S^{3}$. Suppose that $S$ is a $2$-sphere in $S^{3}$ meeting $L$ transversely in $4$ points and dividing $(S^{3},L)$ into two $2$-string tangles $(A, A\cap L)$ and $(B, B\cap L)$. Then if the two tangles $(A, A\cap L)$ and $(B, B\cap L)$ are inseparable, then $L$ is non-split.
\end{enumerate}\label{lem:Nakanishi}
\end{lemma}

\begin{prop}
        Let $(C_1,t_1)$ and $(C_2,t_2)$ be inseparable tangles. If the tangle diagram in Figure \ref{fig:anynonsplit} appears in a diagram of a 2-component link $L$ in such a way that the solid arcs belong to the same link component and the dotted arcs belong to the other component, then $L$ is non-split. \label{prop:guaranteesnonsplit}
\end{prop}
\begin{proof}
    Let $S$ be the 2-sphere represented by the dashed red-blue oval that contains both $(C_1,t_1)$ and $(C_2,t_2)$ on the inside. In Figure \ref{fig:anynonsplit}, the sphere $S$ is represented as being made up of two disks colored gray and blue. We will use Nakanishi's result that tangle sum $(C_1,t_1) \cup (cl(S^3
\backslash C_1), cl(S^3
\backslash C_1) \cap L)$ is a non-split link \cite{nakanishi1981primeness}. The tangle $(C_1,t_1)$ is inseparable by assumption. We will show that the 2-string tangle $(cl(S^3
\backslash C_1), cl(S^3
\backslash C_1) \cap L)$ is also inseparable using part (1) of Lemma \ref{lem:Nakanishi}.

The tangle $(cl(S^3
\backslash C_1), cl(S^3
\backslash C_1) \cap L)$ can be subdivided by the disk $D=cl(S \cap (S^3
\backslash C_1))$ into two tangles. In Figure \ref{fig:anynonsplit}, the disk $D$ is represented as the lighter gray dashed arc on the left side of the figure. One of the tangles after the subdivision is $(C_2,t_2)$, which is inseparable by assumption. Let $(A,t)$ be the tangle to the outside of $S$ in Figure \ref{fig:anynonsplit}. Suppose that there is a disk $D'$ that separates the two arcs of $(A,t)$. Such a disk $D'$ must intersect $D$ since the two points of $\partial t$ on $D$ belong to different components of $t$. Therefore, by
part (1) of Lemma \ref{lem:Nakanishi}, the tangle $(cl(S^3
\backslash C_1), cl(S^3
\backslash C_1) \cap L)$ is also inseparable. By
part (2) of Lemma \ref{lem:Nakanishi}, we conclude that $L$ is non-split.

\end{proof}
\begin{figure}[h]
\centering
    \labellist
\small\hair 2pt

\pinlabel{$t_1$} at 249 85

\pinlabel{$t_2$} at 101 85
\pinlabel{$C_1$} at 291 150
\pinlabel{$C_2$} at 55 150

\endlabellist
\includegraphics[width=4cm]{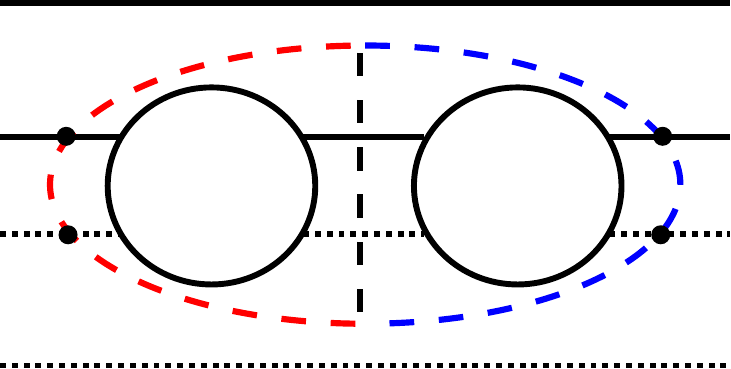}
\caption{This figure shows a 4-string tangle. Let $(C_1,t_1)$ and $(C_2,t_2)$ be inseparable 2-string tangles. If a link contains the 4-string tangle in this figure, where the solid lines belong to the same component, then the link is non-split.}
\label{fig:anynonsplit}
\end{figure}

Theorem 6.3 of \cite{Atapour10} is a special case of Proposition \ref{prop:guaranteesnonsplit}, where the tangles $(C_1,t_1)$ and $(C_2,t_2)$ used by the authors are both the granny tangle. However, with the result below, more tangles such as the ones in \cite{kanenobu2003tangles} can be exploited.

By a tangle diagram bounded by a colored circle, we think of the entire diagram lying on a plane, and we consider the bounded region with respect to that circle. We also remark that the tangle diagram bounded by the red circle and the tangle diagram bounded by the blue circle in Figure \ref{fig:Incomp4} satisfies Proposition \ref{prop:guaranteesnonsplit}, where $(C_1,t_1)$ is the tangle $5_1$ and $(C_2,t_2)$ is the tangle $6_3$ from Kanenobu's table of prime tangles \cite{kanenobu2003tangles}.
\begin{prop}\label{prop:non-split}
    Let $L$ be a 2-component link. Assume that there exists a 2-sphere
$S$ intersecting each component of $L$ in 4 points such that $S$ bounds the tangle in Figure \ref{fig:Incomp4}
on one side, then $L$ is a non-split link.
\end{prop}
\begin{proof}
    Note that the points labeled $a$ and 7 cobound an arc. Similarly each of the pairs $\{b,5\}$, $\{c,6\}$ and $\{d,8\}$ cobound arcs. If $a$ and $b$ belong to the same link component, then $c$ and $d$ belong to the other link component and $L$ is non-split, by Proposition \ref{prop:guaranteesnonsplit} applied to the tangle diagram bounded by the red circle. If $b$ and $c$ belong to the same link component, then $1$ and $2$ belong to the same link component and $3$ and $4$ belong to the other link component. Hence, we can apply Proposition \ref{prop:guaranteesnonsplit} to the tangle diagram bounded by the blue circle to conclude that $L$ is non-split. Finally, if $a$ and $c$ belong to the same link component, we can do a rotation of the middle two arcs of the tangle bounded by the red circle (i.e. the tangle sum of $5_1$ and $6_3$) so that after the rotation $a$ and $c$ become adjacent in the diagram and Proposition \ref{prop:guaranteesnonsplit} can be applied to show $L$ is non-split.
\end{proof}
\begin{prop}\label{Prop:splittingGood}
    The splitting number is a good measure of 2-component spanning link complexity. Furthermore, for this measure, the $L$ in criteria (ii) of Definition \ref{def:goodmeasure} cannot be of the form $L(K)$ for any $K$.
\end{prop}
\begin{proof}
    Let $F(L)$ denote the splitting number of $L$. Since the unlink is a split link, its splitting number is zero satisfying condition (i) of Definition \ref{def:goodmeasure}. Let $\widehat{L}$ be the link in Figure \ref{fig:closedlink} and set $L=(\widehat{L},G_1,G_2)$. Then, $L$ is non-split as a link by Proposition \ref{prop:non-split}. We claim that for any integer $m$ and any $J_1, J_2,...,J_{m+1}\in \mathcal{L}_T^2$, 
$F(J_1\underline{*} L \underline{*} J_2 \underline{*} L \underline{*} J_3 \underline{*}...\underline{*}J_{m}\underline{*}L\underline{*}J_{m+1}) 
\geq mF(L) > 0$. To see this note that a single crossing change at the highlighted crossing in Figure \ref{fig:closedlink} turns $L$ into a split link. Hence, $F(L)=F(\widehat{L})=1$.  By Proposition \ref{prop:non-split}, each instance of $L$ requires at least one crossing change to split the total link. 
\end{proof}

\begin{figure}[h]
\centering
    \labellist
\small\hair 2pt
\pinlabel{$1$} at 240 280
\pinlabel{$2$} at 243 199
\pinlabel{$3$} at 247 132
\pinlabel{$4$} at 280 28
\pinlabel{$5$} at 513 299
\pinlabel{$6$} at 520 252
\pinlabel{$7$} at 531 172
\pinlabel{$8$} at 520 59
\pinlabel{$a$} at 13 289
\pinlabel{$b$} at -9 112
\pinlabel{$c$} at -11 64
\pinlabel{$d$} at 19 30
\pinlabel{$e$} at 169 253
\pinlabel{$f$} at 149 215
\pinlabel{$g$} at 170 95
\pinlabel{$h$} at 149 30

\endlabellist
\includegraphics[width=7cm]{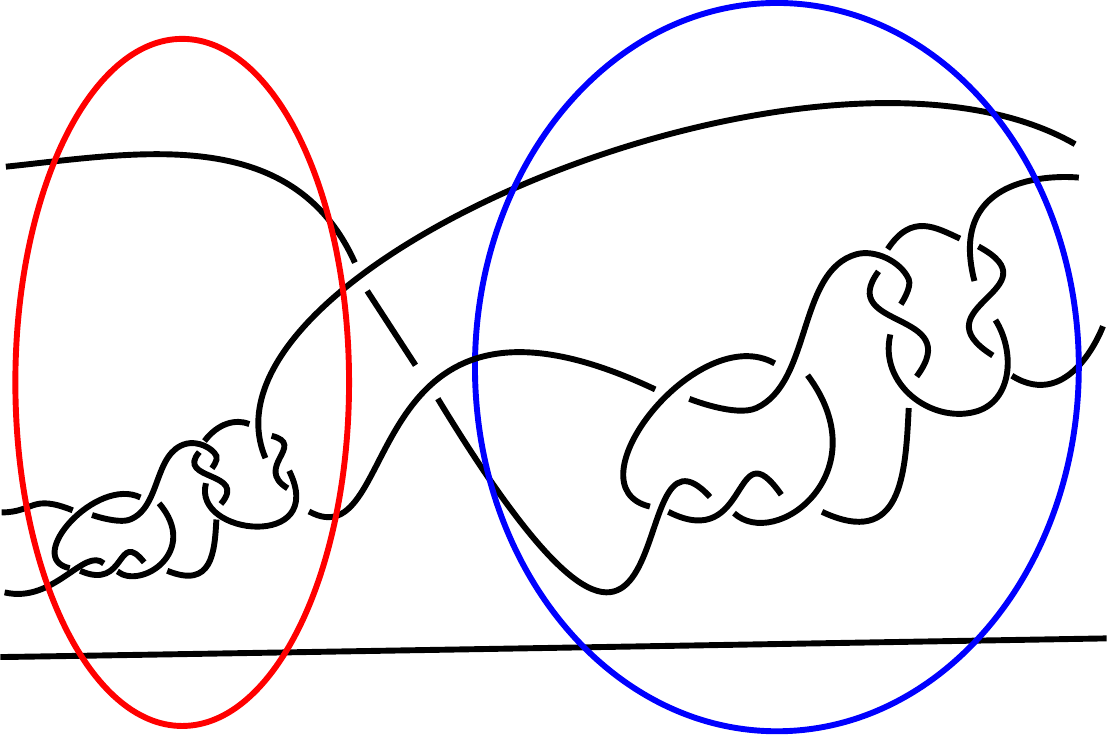}
\caption{If a link contains this tangle, then the link is non-split.}
\label{fig:Incomp4}
\end{figure}

\begin{figure}[h]
\centering
\includegraphics[width=8cm]{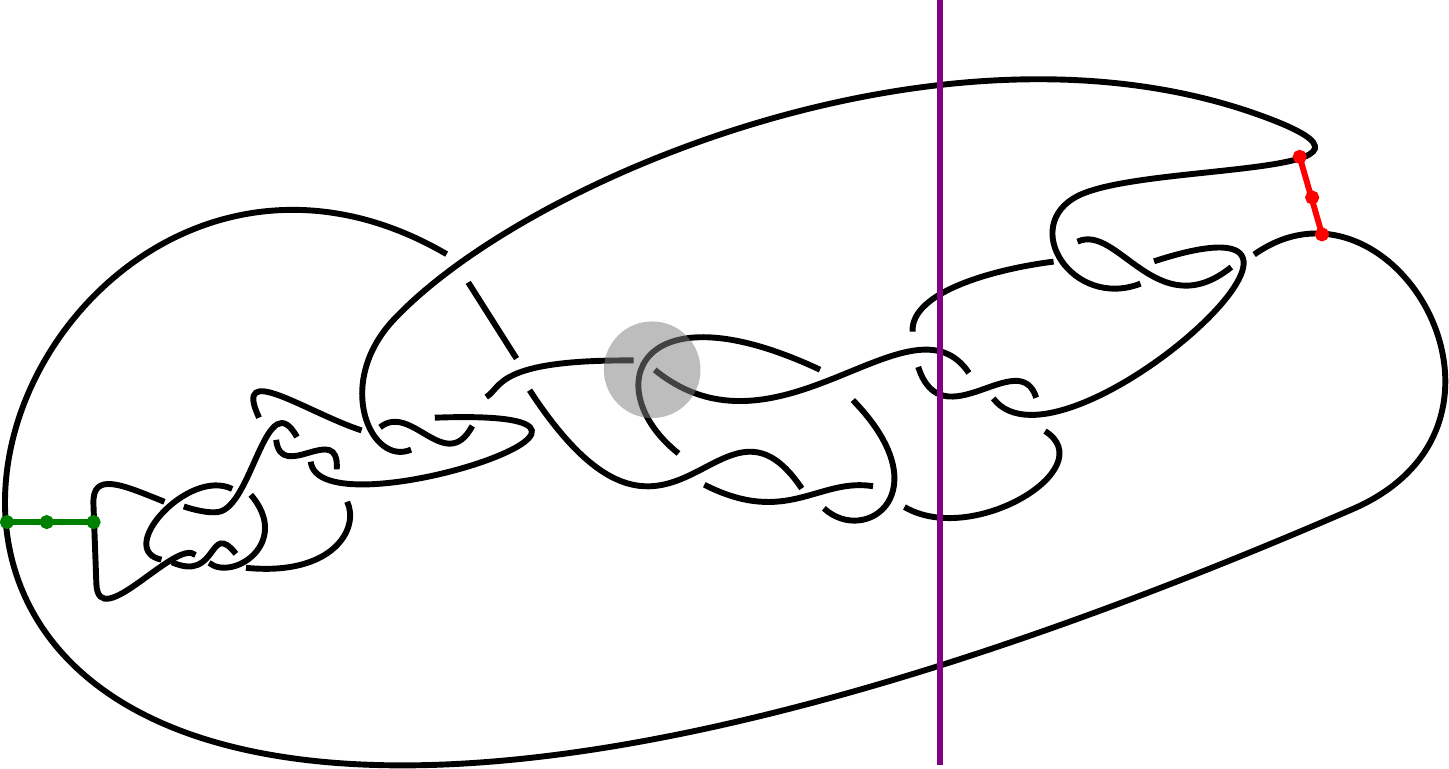}
\caption{The 2-component link $\widehat{L}$ is obtained from the tangle in Figure \ref{fig:Incomp4} by connecting endpoints labeled $a$ with $d$ and then $b$ with $c$. The equal height trunk of this link is 6. The purple line represents a level sphere that intersects the link 6 times, and there are no other level spheres that intersect the corresponding Morse embedding in more than 6 times. Also, changing the highlighted crossing turns the link into a split link.}
\label{fig:closedlink}
\end{figure}

\subsection{Entanglement complexity of  2SAPs.}
\label{sec:2SAPcomplexity}

The pattern theorem for 2SAPS of Theorem \ref{patternthm} implies the following result.

\exprare*
\begin{proof}

Consider $\tube$ and $P_L$ as in the theorem statement. 
Suppose $P_L$ has span $s$. From Theorem \ref{patternthm}, there exist positive $\epsilon_L$ and integer $m_L> (s+1)/\epsilon_L=1/\hat{\epsilon}_L$, such that for sufficiently large $m>m_L$, all but exponentially few size $m$ 2SAPs in $\tube$ contain at least $\lfloor \epsilon_Lm\rfloor$ translates of $P_L$.  
Since $P_L$ has span $s$, at least $\lfloor \epsilon_Lm\rfloor/(s+1)\geq \lfloor \epsilon_L m/(s+1) \rfloor=\lfloor \hat{\epsilon}_L m\rfloor$ of the translates are non-overlapping. (To see this, consider the translate of $P_L$, denoted $\pi_1$, that contains a point with the largest $x$ value and note that at most $s$ other translates of $P_L$ can intersect  $\pi_1$, since the span of $\pi_1$ is $s$. Now ignore $\pi_1$ and any translates of $P_L$ that intersect  $\pi_1$. Of the remaining translates of $P_L$, let $\pi_2$ be the translate that contains a point with the largest $x$ value. Then $\pi_2$ is disjoint from $\pi_1$ and intersects at most $s$ of the remaining translates. We proceed inductively to produce at least $\lfloor \epsilon_Lm\rfloor/(s+1)$  non-overlapping translates of $P_L$.)
 
Thus for sufficiently large $m > m_L$, all but exponentially few  size $m$  2SAPs have link-type of the form $L' = J_1\underline{*} L \underline{*} J_2 \underline{*} L \underline{*} J_3 \underline{*}...\underline{*}J_{\lfloor \hat{\epsilon}_L m\rfloor}\underline{*}L\underline{*}J_{\lfloor \hat{\epsilon}_L m\rfloor +1}$. 
 
 Since $F$ is a good measure, condition (ii) of Definition  \ref{def:goodmeasure} gives that     $F(L') \geq (\lfloor m\hat{\epsilon}_L \rfloor/2) F(L) > F(L)((2m_L)^{-1}m -1)$. 
      Further, taking $m>  2m_L +1$ and $0 <\beta\leq \frac{1}{2m_L} - \frac{1}{2m_L +1}$
      gives $F(L')\geq \beta m F(L)$, so that all but exponentially few sufficiently large size $m$ 2SAPs have $F$-complexity that grows at least linearly in $m$, as $m\to\infty$. 
\end{proof}

Next we establish appropriate choices of $P_L$ and tube dimensions for several classes of good measures of 2-component spanning link complexity. 

\begin{prop}\label{prop:patterns}
There exists patterns $P_L$ and associated  tubes $\tube_{N,M}$ as required for Theorem \ref{thm:exprare} for $F$ equal to splitting number or for $F$  a good measure of 2-component spanning link complexity that satisfies either Theorem \ref{thm:LinkMeasureToGoodMeasure} or Theorem \ref{thm:moregoodmeasures}.  
\end{prop}
\begin{proof}
If $F$ is a good measure of 2-component spanning link complexity that satisfies either Theorem \ref{thm:LinkMeasureToGoodMeasure} or Theorem \ref{thm:moregoodmeasures}, then a local 2SAP knot pattern of a knot $K$ works as $P_L$ in the theorem.  If $F$ satisfies Theorem \ref{thm:LinkMeasureToGoodMeasure},  $K$ can  be any knot, otherwise taking $K$ to be a knot that satisfies condition (2) of Theorem \ref{thm:moregoodmeasures} for $F$ works.  
Given a knot $K$, a local 2SAP knot pattern of $K$ fits in at least $\tube_{N,M}$ where $(N+1)(M+1)> \text{trunk}(K)+4$ by Corollary \ref{cor:equalheight}.

Note further that if $F$ satisfies Theorem \ref{thm:LinkMeasureToGoodMeasure},  $K$ can  be any knot, and hence taking $K=3_1$ will work. Since $3_1$ is meridionally small, Corollary \ref{cor:equalheight} gives that a local knot pattern of $3_1$ exists in $\tube=\tube_{3,1}$ (or larger) and hence all but exponentially few size sufficiently large size $m$ 2SAPs in $\tube$ have $F$-complexity
which exceeds $\displaystyle{F(L(3_1))(\frac{m}{2m_L} -1)}$.

In the case of $F$ equal to splitting number, take $P_L$ to be a 2SAP pattern corresponding to the tangle in Figure \ref{fig:closedlink}.  The tangle in Figure \ref{fig:closedlink} fits in $\tube=\tube_{3,1}$ (or larger) by Proposition \ref{prop:3x1}.  In this case the theorem implies that and all but exponentially few size sufficiently large size $m$ 2SAPs in $\tube$ have splitting number
which exceeds $\displaystyle{(\frac{m}{2m_L} -1)}$.
\end{proof}

As noted previously, the number of prime factors  satisfies the conditions of Theorem \ref{thm:moregoodmeasures} for any non-trivial $K$. Hence given a tube size that admits a local 2SAP knot pattern for some non-trivial $K$, there exists a $\beta >0$ such that  all but exponentially few sufficiently large size $m$ 2SAPs have at least $\beta m$ prime factors.

\section{Characterizing 2SAP Entanglements}

\subsection{Equal Height Trunk}\label{Sec:EHTrunk}

Motivated by the situations where one studies a system of 2SAPs where each component has the same span, we consider an alternate version of the trunk of links. 

The standard height function $h:\mathbb{R}^3\rightarrow \mathbb{R}, h(x,y,z) = y$ gives a decomposition of the 3-space into parallel level planes $h^{-1}(t)$ for $t \in \mathbb{R}$. We will only consider embeddings where $h$ restricted to the image of any such embedding is a Morse function. The \textit{trunk} of the embedding is defined to be
\begin{align*}
    \text{trunk}(l) = \max_{t\in \mathbb{R}}|h^{-1}(t) \cap l|.
\end{align*}
The \textit{trunk} of a link type $L$ is the minimum trunk over all morse embeddings of $L$.

In this paper we will focus on morse embeddings of a $2$-component link where the highest maximum of each component lies at the same height $h_{max}$, the lowest minimum of each component lies at the same height $h_{min}$. Given such an embedding, the \textit{equal height trunk} of the embedding is defined to be
\begin{align*}
    \text{EH-trunk}(l) = \max_{t\in \mathbb{R}}|h^{-1}(t) \cap l|.
\end{align*}
The \textit{equal height trunk} of a link type $L$ is the minimum EH-trunk over all embeddings of $L$ with the restrictions stated above.

\begin{thm}[Theorem 1 of \cite{ishihara2017bounds}] \label{thm:IshiBound} A link type $\mathcal{L}$ can be can be confined in a $M\times N$ tube if and only if
trunk$(\mathcal{L}) < (M + 1)(N + 1)$.
\end{thm}

The following theorem is a straightforward adaptation of Theorem \ref{thm:IshiBound} and its proof. Given an embedding of a 1-manifold $\pi$ in a tube and an integer $k$, a \textit{hinge} $H_k$ is defined to be the set of vertices and edges lying in the intersection of $\pi$ and the $yz$-plane defined by $\{(x,y,z) \;\ | \;\ x=k\}.$

\equalheight*

\begin{proof}
Suppose there is an embedding $\gamma$ of $L$ confined in a $M\times N$ tube in such a way that each component of $\gamma$ has the same span. Then \text{EH-trunk}$(L) \leq$ \text{trunk}$(\gamma) \leq (M+1)(N+1)$. Suppose \text{trunk}$(\gamma) = (M+1)(N+1)$. Then there exists $t\in \mathbb{R}\setminus \mathbb{Z}$ such that $|h^{-1}(t)\cap \gamma|=(M+1)(N+1)$. Let $H_{a}$ be the first hinge below $t$ that contains an edge and let $H_{b}$ be the first hinge above $t$ that contains an edge. If $a=h_{min}$ and $b=h_{max}$, then $L$ is a trivial link, a contradiction. Without loss of generality, suppose $h_{min}<a$. The embedding $\gamma$ can be modified by translating an edge in $H_{b}$ in the negative $x$ direction until it lies in $H_{a}$ (see how Figure \ref{fig:fullsection} changes to Figure \ref{fig:fullsection2}). Note that the translated copy of the edge in $H_{b}$ cannot correspond to an edge in $H_{a}$ as this would imply that the unknotted component of $\gamma$ containing these edges had a different span than the other component of $\gamma$. Since the translated copy of the edge in $H_{b}$ does not correspond to an edge in $H_{a}$, then this translation induces an isotopy taking $\gamma$ to a different embedding $\gamma'$ of $L$ confined in a $M\times N$ tube such that for all $t\in \mathbb{R}\setminus \mathbb{Z}$, $|h^{-1}(t)\cap \gamma'|<(M+1)(N+1)$ whenever $a-1<t<b+1$. We can repeat this process to eliminate all $t$ such that $|h^{-1}(t)\cap \gamma|=(M+1)(N+1)$ and produce an embedding $\gamma'$ of $L$ in a $M\times N$ tube in such a way that each component of $\gamma$ has the same span and $\text{trunk}(\gamma') < (M+1)(N+1)$. Hence, \text{EH-trunk}$(L) < (M+1)(N+1)$.

Suppose that EH-trunk$(L)<(M+1)(N+1)$. In the proof of Theorem \ref{thm:IshiBound}, the authors show that given any morse embedding $\gamma$ of the link $L$ such that trunk$(\gamma)<(M+1)(N+1)$, then $\gamma$ can be embedded in the $M\times N$ tube in such a way that the relative heights of the minima and maxima of $\gamma$ with respect to $h$ are preserved. Suppose $\gamma'$ is an embedding of $L$ such that trunk$(\gamma')=$ EH-trunk$(\gamma')=$ EH-trunk$(L)<(M+1)(N+1)$. Then $\gamma'$ can be embedded in the $M\times N$ tube in such a way that the relative heights of the minima and maxima of $\gamma'$ with respect to $h$ are preserved. In particular, each component of the resulting embedding has the same span.

\end{proof}

\begin{figure}[h]
\labellist
\small\hair 2pt
\pinlabel $H_a$  at 100 -10
\pinlabel $H_b$  at 165 -10
\endlabellist
\centering
\includegraphics[width=5cm]{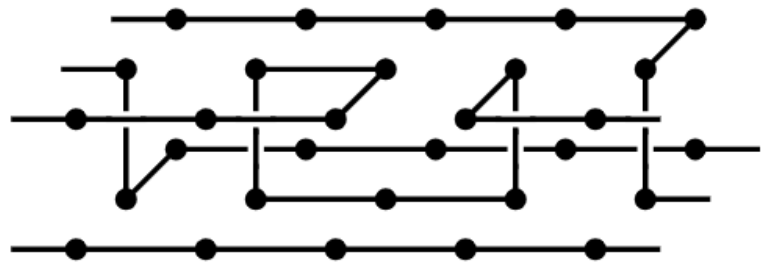}
\caption{A 2SAP with an $(M+1)(N+1)$ section.}
\label{fig:fullsection}
\end{figure}

\begin{figure}[h]
\labellist
\small\hair 2pt
\pinlabel $H_a$  at 100 -10
\pinlabel $H_b$  at 165 -10
\endlabellist
\centering
\includegraphics[width=5cm]{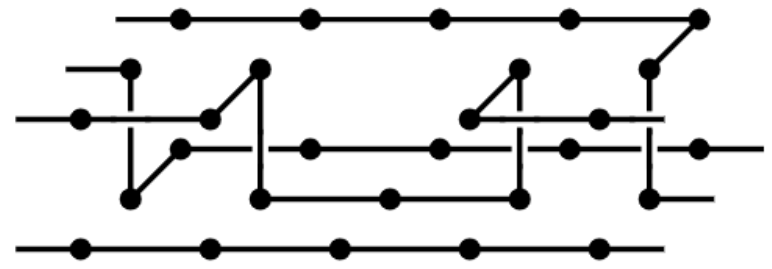}
\caption{After reductions, a nontrivial 2SAP does not contain an $(M+1)(N+1)$ section.}
\label{fig:fullsection2}
\end{figure}

As an application, we can show that the tangle in Figure \ref{fig:Incomp4} can be embedded in an $N\times M$ tube, where $M+N\geq 4.$

\begin{prop}
    The tangle in Figure \ref{fig:Incomp4} can be embedded in a $(3\times 1)$-tube.\label{prop:3x1}
\end{prop}

\begin{proof}
   If the link $L$ in Figure \ref{fig:closedlink} can be embedded in a $(3\times 1)$-tube, we can remove some edges at the leftmost and the rightmost hinges and the tangle in Figure \ref{fig:Incomp4} can also be embedded in a $(3\times 1)$-tube. Figure \ref{fig:closedlink} depicts a level sphere that realizes the equal height trunk of 6. By Theorem \ref{thm:equalheight}, the link $L$ can be embedded in a $(3\times 1)$-tube since $6< (3+1)(1+1).$
\end{proof}

For knots, the concepts of local (1-filament) and non-local (2-filament) knot patterns were defined in \cite[Section 3]{beaton2018characterising}. In \cite[Result $4^*$ of the Supplemental Information]{beaton2018characterising}, some results about the tube sizes that admit non-local and local knot patterns are presented. Here we improve those results under an additional assumption.

\begin{prop}
Suppose that $K\in {\cal L}^1$ embeds in an $(N \times M)$-tube. Suppose also that a trunk minimizing diagram of $K$ contains an arc of the diagram that connects the highest local maximum to the lowest local minimum isolated as shown in Figure \ref{fig:localpattern}, then a local  knot pattern of $K$ also embeds in an $(N \times M)$-tube.\label{prop:fittinglocal}
\end{prop}
\begin{proof}
   Cutting the special arc stated in the statement of the proposition at its ends yields a 2-string tangle that is a local knot pattern. Since the knot itself already fits, the resulting 2-string tangle gives a knot pattern that still embeds in the same tube size.
\end{proof}
\begin{figure}[ht!]
\centering
\includegraphics[width=3cm]{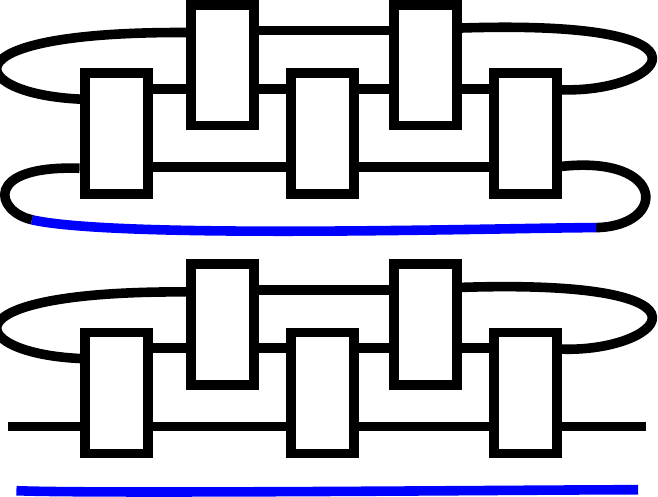}
\caption{Cutting the special arc colored in blue at its ends (in the top diagram) gives a local pattern (bottom diagram). The rectangles represent braid boxes.}
\label{fig:localpattern}
\end{figure}

\begin{figure}[ht!]
\centering
\includegraphics[width=3cm]{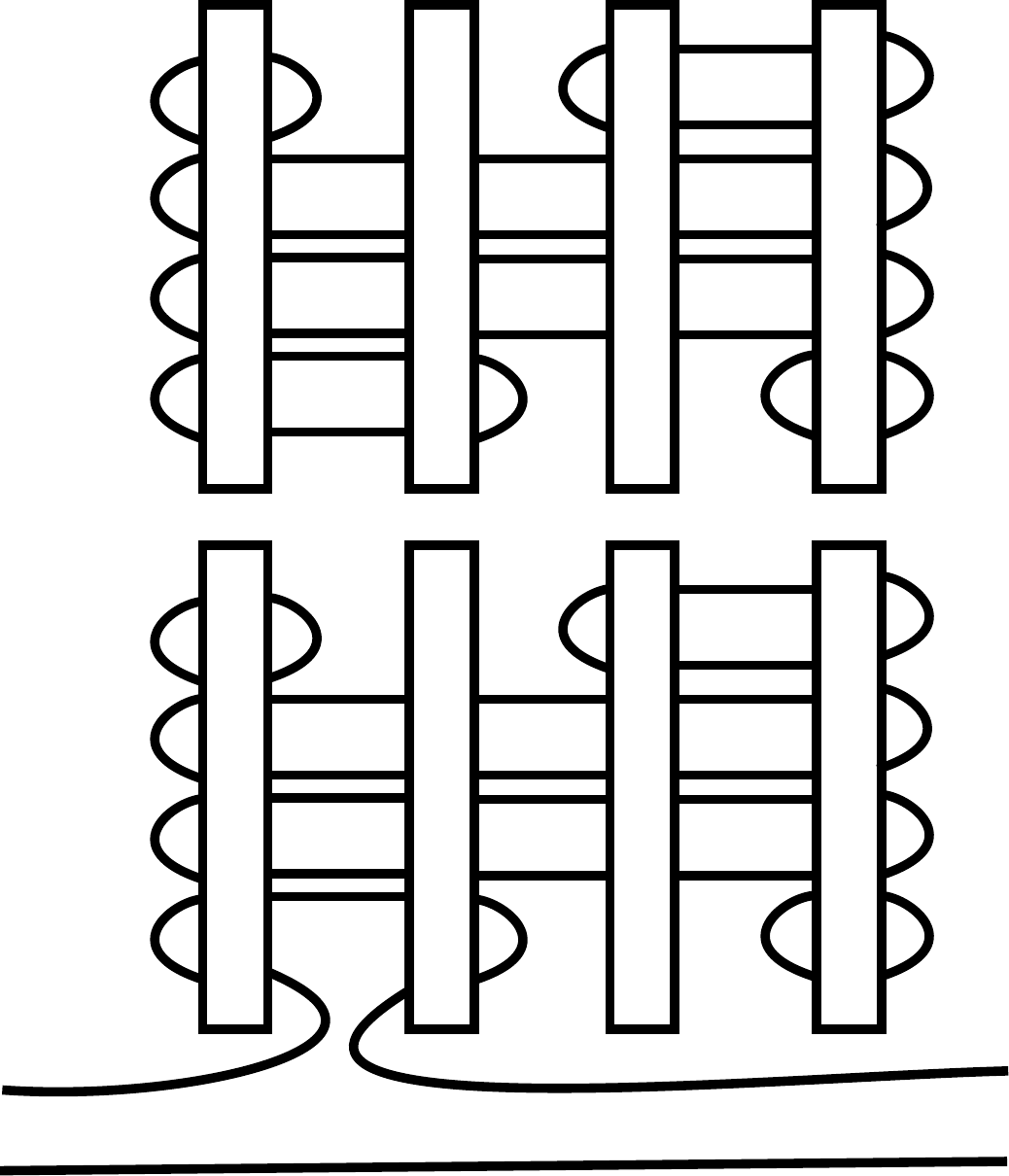}
\caption{A case where trying to create a local pattern (bottom diagram) from the top diagram may need a larger tube size. The rectangles represent braid boxes.}
\label{fig:mayincrease}
\end{figure}

A knot where a local knot pattern may not exist in the same tube size has a schematic picture in Figure \ref{fig:mayincrease}. In this case, removing an arc (from the top diagram) to make a local pattern may involve increasing the trunk of the associated tangle. Therefore, a larger tube size is needed. Examples of knots with trunk minimizing diagrams containing a special arc as in Figure \ref{fig:localpattern} include knots such that a trunk minimizing embedding is in bridge position. In \cite{ozawa2010waist}, the author shows that all meridionally small knots have the property that a trunk minimizing embedding is in bridge position. The class of meridionally small knots includes 2-bridge knots and torus knots. Note that a knot is meridionally small if there exists no essential surface properly embedded in the knot exterior with meridional boundary, for more details see \cite{ozawa2010waist}.

\equalheightplus*
\begin{proof}

    We work with the standard height function $h:\mathbb{R}^3\rightarrow \mathbb{R}, h(x,y,z) = y$. Let $k$ be an embedding for $K$ that realizes trunk$(K)$. Let $a$ denote the height of the highest local maximum of $K$ and $b$ denote the height of the lowest local minimum of $K$. We can form an embedding of $L$ by placing the unknot component in such a way that it has exactly one local maximum at height $a$ and one local minimum at height $b.$ This gives an embedding for $L$ with equal height trunk at most trunk$(K)$ + 2.

    Let $k'$ be an embedding of $L$ that realizes the EH-trunk$(L)$. Because the highest maximum of each component lies at the same height $h_{max}$, and the lowest minimum of each component lies at the same height $h_{min}$, any level plane $h^{-1}(t)$ between $h_{min}$ and $h_{max}$ intersects both components of $k'$. The unknot component contributes at least two intersection points to $|k'\cap h^{-1}(t)|$ for any $t\in (h_{min},h_{max}).$ Additionally, The component isotopic to $K$ contributes at least trunk$(K)$ intersection points to $|k'\cap h^{-1}(t_0)|$ for some regular value $t_0\in (h_{min},h_{max}).$ Therefore, EH-trunk$(L)\geq |k'\cap h^{-1}(t_0)|\geq$ trunk$(K)+2$. Combining this with the bound from the previous paragraph, we have that EH-trunk$(L) =$ trunk$(K)+2$.

    The remaining claims can be proved essentially using the arguments in \cite[Result $4^*$]{beaton2018characterising}. If $L$ embeds as a 2SAP in a $(N\times M)$ tube, then pulling the string as shown in Figure \ref{fig:mayincrease} gives a local 2SAP knot pattern of $K$ in a $((N+1)\times M)$ tube because we have found an embedding with EH-trunk at most $(N+1)(M+1)+2$. 
    
    Suppose that $K$ is a meridionally small knot and that $L=0_1\cup K$ embeds as a 2SAP in a $(N\times M)$ tube. By Proposition \ref{prop:fittinglocal} and the discussion that follows, there is a trunk-minimizing diagram of $K$ that is in bridge position and a corresponding local 2SAP knot pattern of $K$ that embeds in a $(N\times M)$ tube.
\end{proof}
In this paper, we are interested in the equal height trunk such that the trunk of one particular component is minimized. Observe that it is in general not possible to find a position that minimizes that trunk of both components simultaneously. For instance, consider 2-component links, where each component is the unknot. If such a position exists, then the trunk of each component is 2. In particular, the link is a 2-bridge link. However, there exist links where two components are both unknots, but the bridge number of the entire link is arbitrarily large.

\subsection{Proof of Theorem \ref{thm1}}

Before proving Theorem \ref{thm1}, further definitions and results are needed. 
\subsubsection{Satellite Links}\label{Sec:SatLinks}
Roughly speaking, we form a two component satellite link by knotting up two solid tori, each of which contains a knot, into the shape of a different two component link (see Figure \ref{fig:sat}). More precisely, let $L^0=L^0_1\cup L^0_2$ be a 2-component link in the $3$-sphere. Let $\widehat{K_i}$ be a knot that is not contained in a 3-ball in an unknotted solid torus $\widehat{V_i}$. We also require that $\widehat{K_i}$ is not isotopic to the core of $\widehat{V_i}$. Let $\Psi_i$ be a homeomorphism from $\widehat{V_i}$ to a small regular neighborhood of $L^0_i$. Denote the image of $\widehat{V_i}$ by $V_i$, the image of $\widehat{K_i}$ be $K_i$, and the boundary $\partial(V_i)$ of $V_i$ as $T_i$. In this construction, we call $L=K_1 \cup K_2$ a \textit{satellite link} with patterns $(\widehat{V_1},\widehat{K_1}),(\widehat{V_2},\widehat{K_2})$ and \textit{companions} $L^0_1, L^0_2$.

Recall that a simple closed curve $c$ in a surface $\Sigma$ is \textit{separating} if $S\backslash c$ has two connected components. We say $c\subset \Sigma$ is \textit{inessential} if it is separating and a component of $S\backslash c$ is a disk or annulus. The curve $c$ is \textit{essential} if it is not inessential.

In a solid torus, there is exactly one properly embedded disk whose boundary is an essential curve on the torus, up to proper isotopy. We call such a disk a \textit{meridian} disk. We define the \textit{index} of a pattern $(\widehat{V_i},\widehat{K_i})$ to be the minimal geometric intersection number between $\widehat{K_i}$ and a meridian disk of $\widehat{V_i}.$ If we have a satellite link $L=K_1\cup K_2$ so that one of the companions $L^0_1$ is unknotted. Then, the closure of the exterior of a small regular neighborhood of $L_1^0$ is another solid torus containing the knot $L^0_2$. Thus, $(\overline{S^3\backslash \widehat{V_1}},L_2^0)$ is a pattern, and we call the index of this pattern the \textit{dual index} of $L_1^0.$

\begin{figure}[ht!]
\labellist
\small\hair 2pt

\pinlabel {$\Psi_1$}  at 112 422
\pinlabel {$\Psi_2$}  at 812 422
\pinlabel $\widehat{V_1}$  at 203 70
\pinlabel $\widehat{V_2}$  at 740 70
\pinlabel $V_1$  at 303 460
\pinlabel $V_2$  at 710 610
\pinlabel \tiny $\widehat{K_1}$  at 103 189
\pinlabel \tiny $\widehat{K_2}$  at 668 189

\endlabellist
\centering
\includegraphics[width=8cm]{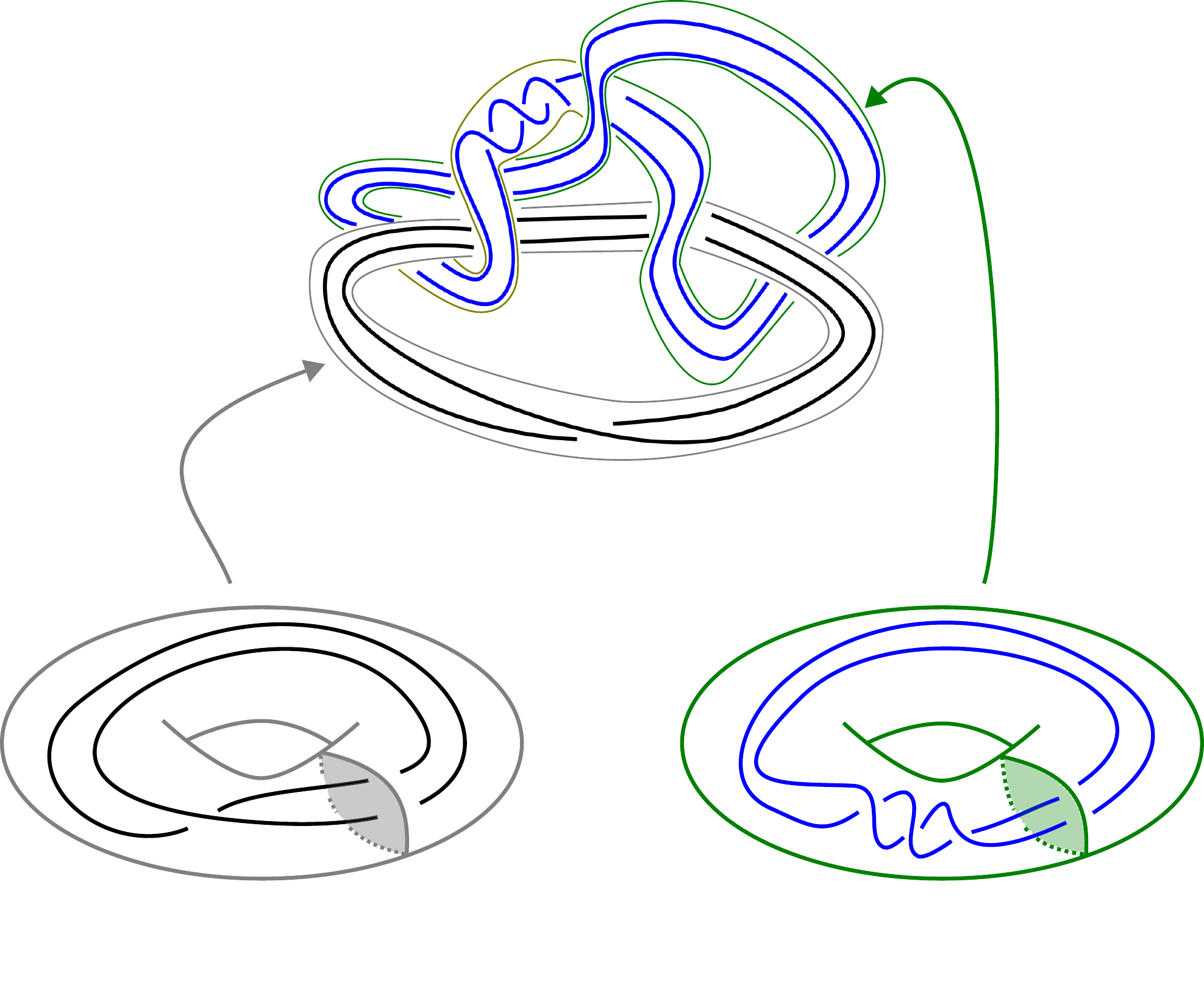}
\caption{A satellite 2-component link is formed by taking two solid tori each containing a knot, and then knotting up those solid tori.}
\label{fig:sat}
\end{figure}

\subsubsection{Morse Positions and Essential Saddles}\label{Sec:Morse}

A surface properly embedded in the exterior of a link is \textit{essential} if it is incompressible, boundary- incompressible, and not boundary parallel. Roughly, an essential surface is special since its topology is faithfully represented in the topology of the ambient 3-manifold. For example, the gray and green tori in Figure \ref{fig:sat} are both essential in the exterior of the link, since neither can be deformed to be parallel to the black or blue knot, respectively. Furthermore, if a curve on the green or gray torus bounds a disk in the link exterior, then such a disk intersects the link. This is in contrast to any torus surface embedded in $S^3$, which always has a meridian disk. Such a meridian disk implies that the torus is compressible and, thus, inessential.

One assumption that makes many topological arguments work is that we can position our links and surfaces in a nice way. For instance, we can define a height function $h:S^3\rightarrow \mathbb{R}\cup \{+\infty, -\infty\}$ that can be restricted to an embedded surface. As a consequence of Morse theory, we can always perturb an embedded surface $\Sigma$ slightly so that $h|_{\Sigma}$ has critical points corresponding to local maxima, local minima and saddles. Moreover, we can assume that all critical points are at different heights. If $h|_{\Sigma}$ has these properties we say $h|_{\Sigma}$ is a \textit{Morse function} and $\Sigma$ is a \textit{Morse embedding}. For additional details, see \cite{milnor1963morse}.

Given a surface $\Sigma$ and Morse function $h|_{\Sigma}$ the set $h^{-1}(r)\cap \Sigma$ is a level set consisting of embedded circles when $r$ is a regular value. Additionally, $h^{-1}(c)\cap \Sigma$ is a collection of circles and a point when $c$ corresponds to a maximum or minimum, and $h^{-1}(c)\cap \Sigma$ is a collection of circles and a $S^1\vee S^1$ when $c$ corresponds to a saddle. Thus, the level sets of $h|_{\Sigma}$ give rise to a foliation $\mathcal{F}_\Sigma$ of $\Sigma$ where the leaves of the foliation consist of circles, points and copies of $S^1\vee S^1$.

Suppose that $T$ is an embedded torus bounding a solid torus $V$ in $S^3$. When $c$ is a critical value of $h|_{T}$ corresponding to a saddle of $T$ we can classify the saddle $S=s_1\vee s_2\subset h^{-1}(c)\cap T$ based on the topology of the loops $s_1$ and $s_2$ in $\Sigma$. If either $s_1$ or $s_2$ is inessential in $T$, then $S$ is called an \textit{inessential saddle}. A saddle is \textit{essential} if it is not inessential.  We can find circles $c_1,c_2$, which are parallel in $T$ to $s_1,s_2$, respectively in a level sphere $F$ slightly lower or higher than $h^{-1}(c)$. Then, $c_1 \cup c_2$ bounds an annulus on $F$. If a collar neighborhood of $c_1\cup c_2$ in this annulus is contained in $V$, then we call $S$ a \textit{nested saddle}. See Figure \ref{fig:nested} for an example of a nested saddle.

\begin{figure}[h]
\centering
\includegraphics[width=3cm]{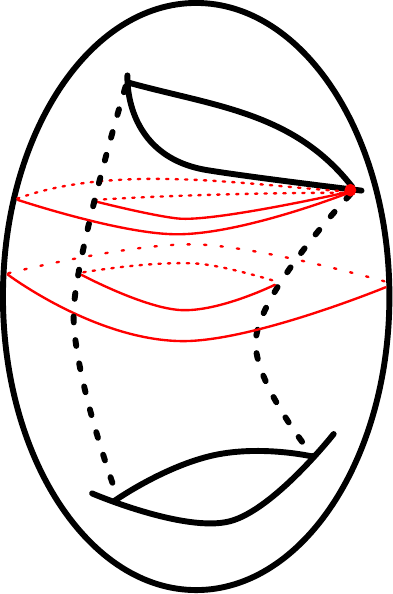}
\caption{A nested saddle.}
\label{fig:nested}
\end{figure}
\subsubsection{Proof of Theorem \ref{thm1}}
In this section we let $L=K_1 \cup K_2$ be a 2-component satellite link with patterns $(\widehat{V_1},\widehat{K_1}),(\widehat{V_2},\widehat{K_2})$ and companions $L^0_1, L^0_2$. Recall from Section \ref{Sec:SatLinks} that $V_i$ is the image of $\widehat{V_1}$ in $S^3$. We additionally assume that $K_1$ and $L^0_1$ are unknots. In particular, the dual index $\omega_1$ of $L^0_1$ exists. 

Given a space $A$, we will let $|A|$ denote the number of connected components of $A$.

We begin with a sequence of lemmas that culminates in the proof of Theorem \ref{thm1}. The authors suspect that the proof of the following lemma already exists in the literature, but we provide the proof here also for completeness.

\begin{lemma}\label{lem:meridianOutside}
    If $D$ is a meridian disk for $\partial V_1$ in $\overline{S^3\backslash V_1},$ then $|D\cap K_2|\geq \omega_1\sigma_2.$\label{lem:claim1}
\end{lemma}
\begin{proof}
    Choose $D$ a meridian disk for $\partial V_1$ in $\overline{S^3\backslash V_1}$ such that $D$ minimizes the lexicographical complexity $(|D\cap K_2|,|D \cap \partial V_2|)$ over all such meridian disks. By transversality, $D\cap \partial V_2$ is a collection of loops. Suppose there is a loop in $D\cap \partial V_2$ that is inessential in $\partial V_2$. Consider such a loop $\gamma\subset D\cap \partial V_2$ that is innermost in $\partial V_2.$ This means that $\gamma$ bounds a disk $E$ in $\partial V_2$ with $int(E)\cap D=\emptyset.$ We surger $D$ along $E$ to obtain meridian disk $D^*$ such that $|D^*\cap \partial V_2| < |D\cap \partial V_2|$ and $|D^*\cap K_2| \leq |D\cap  K_2|$. This violates the minimality of $D$ with respect to the lexicographical complexity. Hence, all loops in $D\cap \partial V_2$ are essential in $\partial V_2$. Consequently, all loops of $D\cap \partial V_2$ are meridians of $\partial V_2$, and $D\cap \partial V_2$ cuts $\partial V_2$ into $j$ annuli $A_1, A_2, \dots , A_j$ where $j=|D\cap \partial V_2|$.

   We say one of the annuli $A_i$ is \emph{boundary-nested} if the two loops of $\partial A_i$ cobound an annulus embedded in $D$. Suppose there exists a boundary-nested $A_i$. Let $\partial A_i=\gamma_1 \cup \gamma_2\subset D$ such that $\gamma_1$ is outermost and bounds a disk $E_1$ in $int(D)$ and $\gamma_2$ is innermost and bounds a disk $E_2\subset E_1\subset int(D)$. The disk $D^*=(D\setminus E_2)\cup A_i \cup E_1$ is a meridian disk for $\partial V_2$ such that, after a small isotopy that pushes $D^*$ off of $A_i$, $|D^*\cap \partial V_2| < |D\cap \partial V_2|$ and $|D^*\cap K_2| \leq |D\cap  K_2|$, violating the minimality of $D$. Hence, none of the annuli $A_1, A_2, \dots , A_j$ are boundary-nested. 

   Every loop  $\gamma_i \subset D\cap \partial V_2$ bounds a unique disk $E_i\subset D$. If a collar of $\gamma_i$ in $E_i$ is contained in $V_2$ we label $\gamma_i$ with ``$I$''. If a collar of $\gamma_i$ in $E_i$ is contained in $S^3\setminus V_2$ we label $\gamma_i$ with ``$O$''. Note that if $A_i$ is not boundary-nested, then the labels on the boundary components of $A_i$ must match. Furthermore, if there is nesting among the loops of $D\cap \partial V_2$ in $D$, then some loops of $D\cap \partial V_2$ are labeled $I$ and some loops are labeled $O$. Hence, there must exist an annulus $A_j$ with one boundary component labeled $I$ and the other labeled $O$. However, $A_j$ must be boundary-nested, a contradiction. Thus, there is no nesting among the the loops of $D\cap \partial V_2$ in $D$. Consequently, $V_2$ intersects $D$ in a collection of $j$ meridian disks for $V_2$. Then $|D\cap K_2|\geq j\sigma_2\geq \omega_1\sigma_2.$

\end{proof}

Recall the definitions of $\mathcal{F}_{V_1}$ and essential saddle given in Section \ref{Sec:Morse}.

\begin{lemma}\label{lem:esssaddleexist}
    $\mathcal{F}_{V_1}$ has an essential saddle. Furthermore, if $S$ is the height of a highest essential saddle, then $h^{-1}(p)\cap \partial V_1$ is a collection of inessential curves in $\partial V_1$ for all regular values $p>S.$
\end{lemma}
\begin{proof}
    Assume that $h|_{\partial V_1}$ is Morse and all critical points are at distinct levels. Let $M=max(h(\partial V_1))$. There exists $\varepsilon > 0$ such that for all regular values $c > M-\varepsilon,$ $h^{-1}(c)\cap \partial V_1$ is a (possibly empty) collection of loops, all of which are inessential in $\partial V_1.$ Note that we get a contradiction if $h^{-1}(c)\cap \partial V_1$ is a collection of inessential loops in $\partial V_1$ for all regular values $c$ because then $\partial V_1$ is homeomorphic to a sphere.

    Define $S$ to be the supremum of all real values $c$ such that $h^{-1}(c)\cap \partial V_1$ contains an essential loop of intersection. Consequently, $S$ is a critical value of $h|_{\partial V_1}.$ Since minima or maxima of $h|_{\partial V_1}$ cannot create or destroy essential loops of intersection, $S$ corresponds to a saddle. Let $\sigma$ be the saddle at height $S$. The boundary of a closed regular neighborhood of $\sigma$ in $\partial V_1$ is a collection of three loops $c_1$, $c_2$ and $c_3$, as depicted in Figure \ref{fig:sadd}. 
    
    If both $c_1$ and $c_2$ are inessential loops in $\partial V_1$, then  $c_3$ is forced to be inessential in $\partial V_1$. By definition of $S$, all loops of $h^{-1}(S+\epsilon)\cap \partial V_1$ are inessential in $\partial V_1$ for small $\epsilon$. Moreover, all loops of $h^{-1}(S+\epsilon)\cap \partial V_1$ are isotopic in $\partial V_1$ to loops in $h^{-1}(S-\epsilon)\cap \partial V_1$ with the possible exceptions of $c_1$, $c_2$ and $c_3$. Since each of $c_1$, $c_2$ and $c_3$ are inessential then all loops in $h^{-1}(S-\epsilon)\cap \partial V_1$ are inessential in $\partial V_1$, which contradicts the definition of $S$. Thus, at least one of $c_1$ or $c_2$ must be essential in $\partial V_1$.

    Without loss of generality, suppose $c_2$ is essential in $\partial V_1$. Then $\sigma$ cannot be a type I saddle as in Figure \ref{fig:sadd}, since the definition of $S$ would imply $c_2$ is inessential. If $\sigma$ is a type II saddle, then, by the definition of $S$, $c_3$ is inessential in $\partial V_1$. If $c_1$ is inessential, then $c_2$ would be forced to be inessential, See Figure \ref{fig:iness}, which is impossible. Hence, both $c_1$ and $c_2$ are essential in $\partial V_1$ and $\sigma$ is an essential type II saddle.
\end{proof}
\begin{figure}[ht!]
\labellist
\small\hair 2pt

\pinlabel {$c_3$}  at 12 22
\pinlabel $c_2$  at 203 93
\pinlabel $c_1$  at -3 93
\pinlabel $c_3$  at 463 103
\pinlabel {$c_1$}  at 271 20
\pinlabel {$c_2$}  at 481 20
\endlabellist
\centering
\includegraphics[scale=0.65]{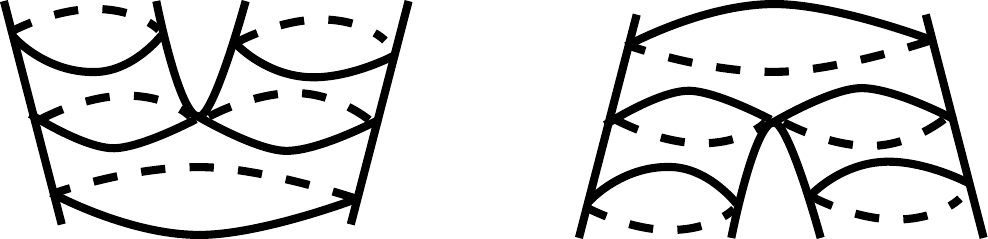}
\caption{(Left) A Type I saddle. (Right) A Type II saddle.}
\label{fig:sadd}
\end{figure}

\begin{figure}[ht!]
\labellist
\small\hair 2pt
\pinlabel $c_2$  at 209 43
\endlabellist
\centering
\includegraphics[scale=0.65]{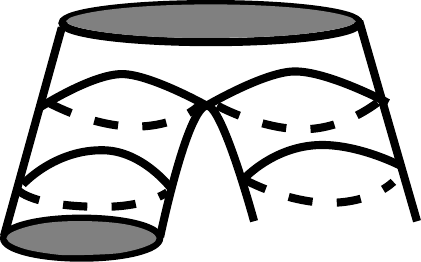}
\caption{If $c_1$ and $c_3$ each bounds a disk in $\partial V_1,$ then $c_2$ does also.}
\label{fig:iness}
\end{figure}

\begin{lemma}\label{lem:goesupmustcomedown}
   Fix $j \geq \beta(K_1)$, where $\beta(K_1)$ is the bridge number of $K_1$ in $S^3.$ Suppose that the index $\sigma_1$ of $(\widehat{V_1},\widehat{K_1})$ is strictly greater than $j$. Additionally, suppose that $\ell'$ is an EH-trunk minimizing embedding of $L$ subject to the constraint that the $K_1$ component of $\ell'$ has $j$ maxima. Let $\sigma$ be the highest essential saddle. Then, $\sigma$ is nested with respect to $V_1$.
\end{lemma}
\begin{proof}
    Let $S$ be the height of $\sigma=s_1\vee s_2.$ Let $D_i$ ($i=1,2)$ be the level disks in $h^{-1}(S)$ bounded by $s_i$ such that $int(D_1)\cap s_2=\emptyset$ and $int(D_2)\cap s_1=\emptyset$. Suppose that $\sigma$ is not nested with respect to $V_1$. It may be that $\partial V_1$ intersects $D_i$ in the interior many times. By transversality, we can assume that $int(D_i)\cap \partial V_1$ is a collection of simple closed curves. 
    
    We first consider $int(D_1)\cap \partial V_1$. Among all these curves, suppose there exists one loop $\gamma$ that is essential in $\partial V_1$. It follows that there exists $\varepsilon > 0$ such that $\partial V_1 \cap h^{-1}(S+\varepsilon)$ contains a curve that is essential in $\partial V_1$. This contradicts Lemma \ref{lem:esssaddleexist}. Therefore, $int(D_1)\cap \partial V_1$ is a collection of loops all of which are inessential in $\partial V_1$. Among them, consider an inessential loop $\gamma$ innermost in $\partial V_1.$  Then, $\gamma = \partial E,$ where $E$ is an embedded disk in $\partial V_1$ such that $int(E) \cap D_1 = \emptyset.$ We surger $D_1$ along $E$ to create a new embedded disk $D^*$ with $|D^*\cap \ell'| \leq |D_1\cap \ell'|$ and $|D^* \cap \partial V_1|<|D_1 \cap \partial V_1|.$ We repeat this process to produce an embedded disk $D^{**}$ such that $int(D^{**})\cap \partial V_1=\emptyset, \partial D^{**}=\partial D_1,$ and $|D^{**}\cap \ell'|\leq |D_1\cap \ell'|.$ Since $\sigma$ is not nested, we know a collar of $\partial D$ in $D$, and thus a collar of $\partial D^{**}$ in $D^{**}$, is contained in $V_1$. Since $int(D^{**})\cap \partial V_1=\emptyset$ then $D^{**}$ is a meridian for $V_1$.

    By the definition of the index, $|D^{**}\cap K_1|\geq \sigma_1$. We conclude that $|D_1\cap K_1|\geq \sigma_1$. The same argument applies to show that $|D_2\cap K_1|\geq \sigma_1$ as well. Since $\sigma$ is the highest saddle, we have a knee-shaped component of $V_1\setminus \sigma$ above $\sigma$, where at least $\sigma_1$ strands enter and must exit $D_1 \cup D_2$ (see Figure \ref{fig:knee}). This means that $|K_1\cap h^{-1}(S)|\geq 2\sigma_1$ and the $K_1$ component of $\ell'$ has at least $\sigma_1$ local maxima, contradicting our assumption that $k_1'$ has $j<\sigma_1$ local maxima.

\end{proof}

\begin{figure}[ht!]
\labellist
\small\hair 2pt
\pinlabel $s_1$  at 1 53
\pinlabel $s_2$  at 199 53
\endlabellist
\centering
\includegraphics[scale=0.65]{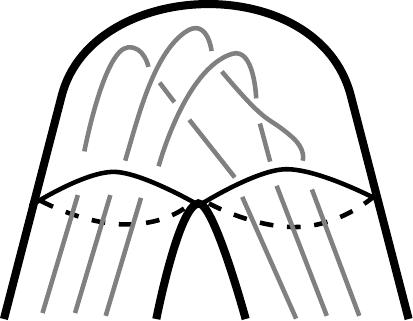}
\caption{A knee component.}
\label{fig:knee}
\end{figure}

\begin{lemma}\label{lem:OmegaSigmaBound}
    Fix $j \geq \beta(K_1)$, where $\beta(K_1)$ is the bridge number of $K_1$ in $S^3.$ Suppose that the index $\sigma_1$ of $(\widehat{V_1},\widehat{K_1})$ is strictly greater than $j$. Additionally, suppose that $\ell'$ is an EH-trunk minimizing embedding of $L$ subject to the constraint that the $K_1$ component of $\ell'$ has $j$ maxima. Let the highest essential saddle of $\mathcal{F}_{V_1}$ lie in the level sphere $h^{-1}(r).$ Then, $|L\cap h^{-1}(r)|\geq 2\omega_1\sigma_2+2$.
\end{lemma}
\begin{proof}
    The argument is similar to Lemma \ref{lem:goesupmustcomedown}, but now we get to use the information on $\omega_1.$ Let $S=s_1\vee s_2\subset h^{-1}(r)$ be the highest essential saddle in $\mathcal{F}_{V_1}$ Let $D_i$ ($i=1,2)$ be the disks bounded by $s_i$ in $h^{-1}(r)$ such that $int(D_1)\cap s_2=\emptyset$ and $int(D_2)\cap s_1=\emptyset$. By transversality, we can assume that $int(D_i)\cap \partial V_1$ is a collection of simple closed curves. Performing the same surgery argument as in the second paragraph of the proof of Lemma \ref{lem:goesupmustcomedown}, we can produce an embedded disk $D^{**}$ such that $int(D^{**})\cap \partial V_1=\emptyset, \partial D^{**}=\partial D_1,$ and $|D^{**}\cap L|\leq |D_1\cap L|.$ By Lemma \ref{lem:goesupmustcomedown}, $S$ is nested with respect to $V_1$. Hence, $D^{**}$ is a meridian disk for $\overline{S^3\setminus V_1}$. By Lemma \ref{lem:meridianOutside}, $|D^{**}\cap K_2|\geq \omega_1\sigma_2$. Consequently, $|D_1\cap K_2|\geq \omega_1\sigma_2$. By an equivalent argument applied to $D_2$, we also conclude that $|D_2\cap K_2|\geq \omega_1\sigma_2$. Additionally, since $\ell'$ is an EH-trunk minimizing embedding of $L$, $|h^{-1}(r)\cap K_1|\geq 2$. All together, we conclude that $|h^{-1}(r)\cap L|=|h^{-1}(r)\cap \ell'|\geq 2\omega_1\sigma_2+2$.

\end{proof}

\begin{remark}
    If EH-trunk is replaced by the usual version of trunk in the statement of the previous theorem, then the conclusion becomes $|L\cap h^{-1}(r)|\geq 2\omega_1\sigma_2$ (the +2 is missing).
\end{remark}

Putting the lemmas together, we are ready to prove our main theorem.
\restrictedtrunk*

\begin{proof}
 Let $\ell'=k_1'\cup k_2'$ be an embedding that minimizes the EH-trunk of $L$ subject to the constraint the trunk of $k_1'=j$.  Because $V_1$ is unknotted, it follows that $S^3\backslash V_1$ is another solid torus. We denote the dual index of $L_1^0$ by $\omega_1$ and the index of $(\widehat{V_2},\widehat{K_2})$ by $\sigma_2.$ Let $\mathcal{F}_{\partial V_1}$ be the foliation on $\partial V_1$ induced by the height function realizing the embedding $\ell'$. 

    By Lemma \ref{lem:meridianOutside}, any meridian disk $D$ for $\partial V_1$ in $\overline{S^3\backslash V_1}$ satisfies $|D\cap k_2| \geq \omega_1\sigma_2.$ By Lemma \ref{lem:esssaddleexist}, $\mathcal{F}_{V_1}$ has an essential saddle. Furthermore, if $r$ is the height of a highest essential saddle for $\mathcal{F}_{\partial V_1}$, then $h^{-1}(p)\cap \partial V_1$ is a collection of inessential curves in $\partial V_1$ for all regular values $p>r.$ By Lemma \ref{lem:goesupmustcomedown}, the saddle $r$ is nested in $V_1$. By Lemma \ref{lem:OmegaSigmaBound}, $|\ell'\cap h^{-1}(r)|\geq 2\omega_1\sigma_2+2$.
\end{proof}

\subsection{Examples and computations}

\begin{example}
    Consider the example in Figure \ref{fig:1maximum}. The example has the Whitehead link $5^2_1$ as the pattern. It has $\sigma_1=2$, $\sigma_2=2$, and $\omega_1=2$. This example also appeared in Kodani's work on bridge number of links subject to the constraint that one of the components has certain number of local maxima \cite{kodani2013new}. Note that the EH-trunk of this knot is at most 8.  By our theorem, the EH-trunk subject to the constraint that the unknotted component has one local maximum is at least $2+2(2\cdot 2)=10$.\label{exmpWhPat}
\end{example}

\begin{figure}[ht!]
\centering
\includegraphics[width=7cm]{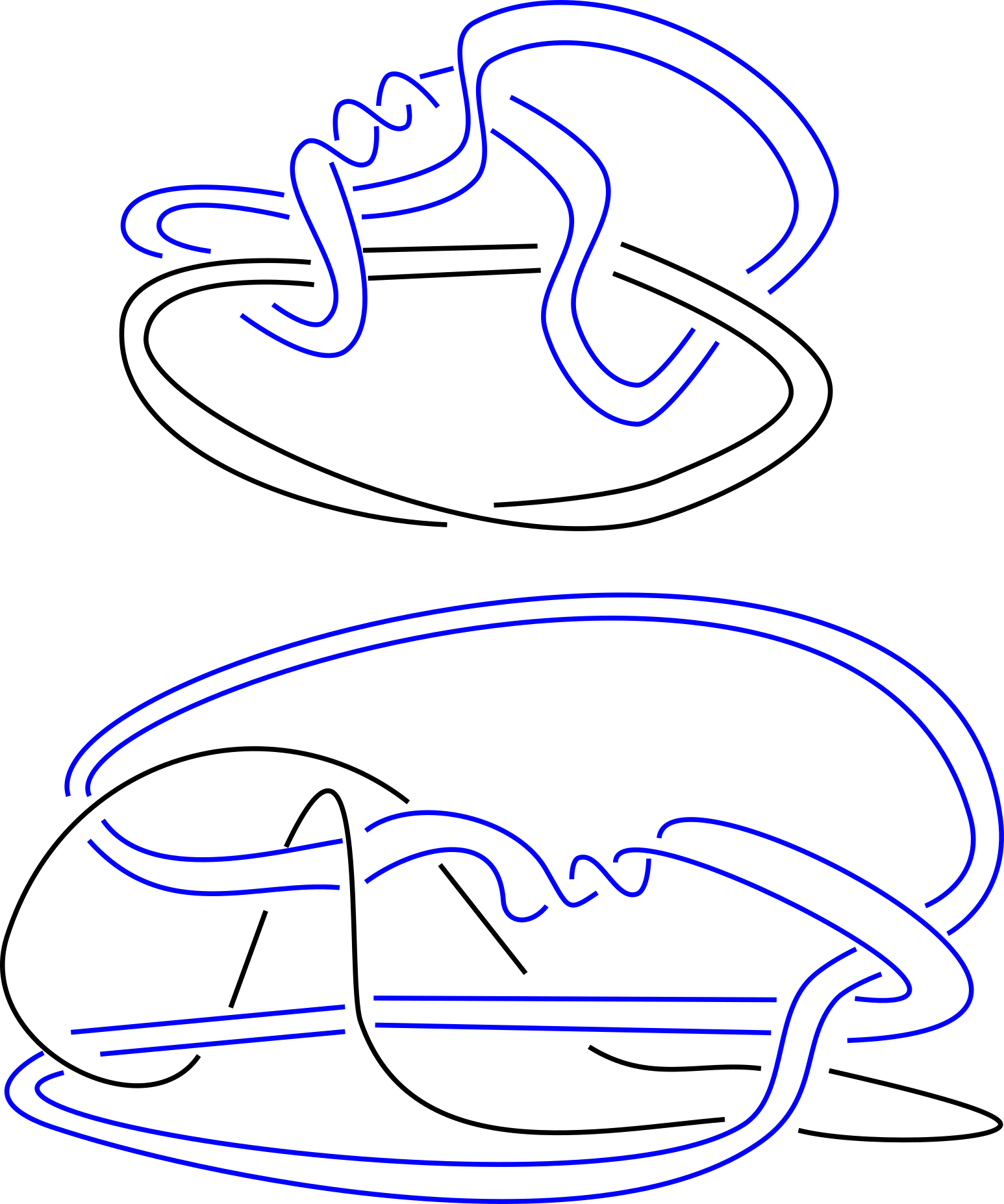}
\caption{Two morse embeddings of the same link type. By the proof of Theorem \ref{thm:IshiBound}, the top embedding can be realized as a 2SAP in a $(2\times 2)$-tube. However, the bottom embedding has trunk $10$ and cannot be realized as a 2SAP in a $(2\times 2)$-tube.}
\label{fig:1maximum}
\end{figure}
\begin{example}
    Consider the example in Figure \ref{fig:scnd}. The example has the link $L6a1$ as the pattern. It has $\sigma_1=2$, $\sigma_2=3$, and $\omega_1=2$. This example also appeared in Kodani's \cite{kodani2013new}.
     Note that the EH-trunk of this knot is at most 10.  By our theorem, the EH-trunk subject to the constraint that the unknotted component has one local maximum is at least $2+2(2\cdot 3)=14$.\label{exmpl6a1Pat}
\end{example}

\begin{figure}[ht!]
\centering
\includegraphics[width=6.5cm]{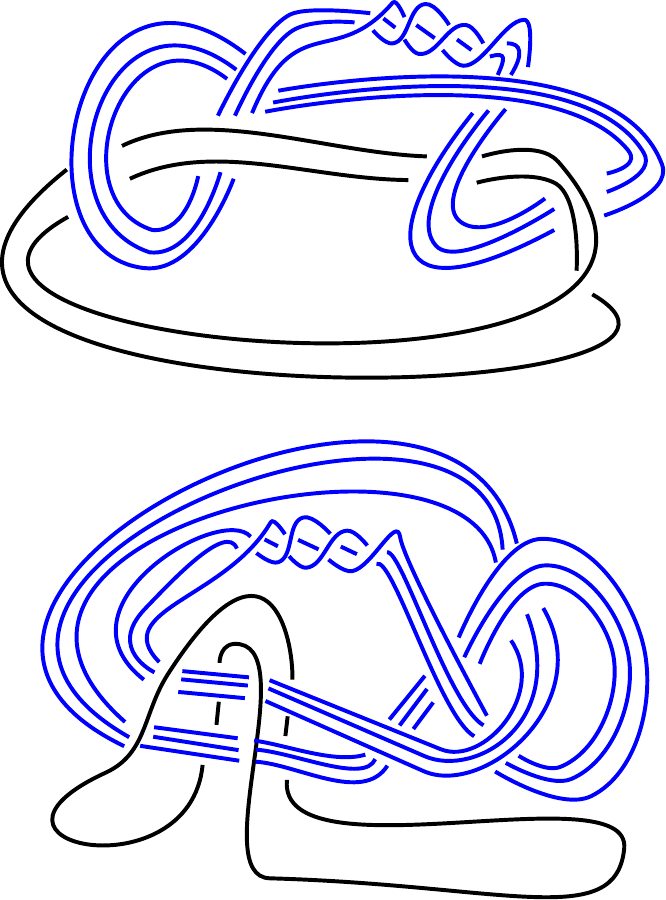}
\caption{Two morse embeddings of the same link type. By the proof of Theorem \ref{thm:IshiBound}, the top embedding can be realized as a 2SAP in a $(2\times 3)$-tube. However, the bottom embedding has trunk $14$ and cannot be realized as a 2SAP in a $(2\times 3)$-tube.}
\label{fig:scnd}
\end{figure}

The above examples have potential applications to experiments involving linked DNA translocating through a nanopore \cite{COW17, COW20, RK23}.  For example,  if the black component of the link in Figure \ref{fig:1maximum} always intersects the plane of the nanopore in at most two points, then at some point there will be at least 10 strands of the DNA trying to simultaneously translocate through the pore, as in the bottom conformation. On the other hand, without that constraint, there may only need to be 8 strands of DNA translocating simultaneously, as in the top conformation. 

\section*{Acknowledgments}

CES acknowledges the support of the Natural Sciences and Engineering Research Council of Canada (NSERC) [funding reference number RGPIN-2020-06339]. PP acknowledges the Pacific Institute for the Mathematical Sciences for its support. RB is partially supported by NSF grant DMS-2424734.

\printbibliography

@article{ishihara2017bounds,
  title={Bounds for minimum step number of knots confined to tubes in the simple cubic lattice},
  author={Ishihara, Kai and Pouokam, Maxime and Suzuki, Atsumi and Scharein, Robert and Vazquez, Mariel and Arsuaga, Javier and Shimokawa, Koya},
  journal={Journal of Physics A: Mathematical and Theoretical},
  volume={50},
  number={21},
  pages={215601},
  year={2017},
  publisher={IOP Publishing}
}

@inproceedings{kodani2013new,
  title={A new bridge index for links with trivial knot components},
  author={Kodani, Yoriko},
  booktitle={Mathematical Proceedings of the Cambridge Philosophical Society},
  volume={154},
  number={2},
  pages={279--286},
  year={2013},
  organization={Cambridge University Press}
}

@inproceedings{schultens2003additivity,
  title={Additivity of bridge numbers of knots},
  author={Schultens, Jennifer},
  booktitle={Mathematical Proceedings of the Cambridge Philosophical Society},
  volume={135},
  number={3},
  pages={539--544},
  year={2003},
  organization={Cambridge University Press}
}

@book{steele1997probability,
  title={Probability theory and combinatorial optimization},
  author={Steele, J Michael},
  year={1997},
  publisher={SIAM}
}

@article{blair2010bridge,
  title={Bridge number and Conway products},
  author={Blair, Ryan C},
  journal={Algebraic \& Geometric Topology},
  volume={10},
  number={2},
  pages={789--823},
  year={2010},
  publisher={Mathematical Sciences Publishers}
}

@inproceedings{delbruck1962mathematical,
  title={Mathematical problems in the biological sciences},
  author={Delbruck, M},
  booktitle={Proc. Symp. Appl. Math.},
  volume={14},
  pages={55},
  year={1962}
}

@article{frisch1961chemical,
  title={Chemical topology},
  author={Frisch, Harry L and Wasserman, Edel},
  journal={Journal of the American Chemical Society},
  volume={83},
  number={18},
  pages={3789--3795},
  year={1961},
  publisher={ACS Publications}
}

@article{sumners1988knots,
  title={Knots in self-avoiding walks},
  author={Sumners, DeWitt and Whittington, SG},
  journal={Journal of Physics A: Mathematical and General},
  volume={21},
  number={7},
  pages={1689},
  year={1988},
  publisher={IOP Publishing}
}

@article{orlandini1994random,
  title={Random linking of lattice polygons},
  author={Orlandini, Enzo and Van Rensburg, EJ Janse and Tesi, MC and Whittington, SG},
  journal={Journal of Physics A: Mathematical and General},
  volume={27},
  number={2},
  pages={335},
  year={1994},
  publisher={IOP Publishing}
}

@article{tesi1998topological,
  title={Topological entanglement complexity of polymer chains in confined geometries},
  author={Tesi, Maria Carla and Van Rensburg, EJ Janse and Orlandini, Enzo and Whittington, Stuart G},
  journal={Topology and Geometry in Polymer Science},
  pages={135--157},
  year={1998},
  publisher={Springer}
}

@article{beaton2018characterising,
  title={Characterising knotting properties of polymers in nanochannels},
  author={Beaton, Nicholas R and Eng, Jeremy W and Ishihara, Kai and Shimokawa, Koya and Soteros, Christine E},
  journal={Soft matter},
  volume={14},
  number={28},
  pages={5775--5785},
  year={2018},
  publisher={Royal Society of Chemistry}
}

@article{ozawa2010waist,
  title={Waist and trunk of knots},
  author={Ozawa, Makoto},
  journal={Geometriae Dedicata},
  volume={149},
  number={1},
  pages={85--94},
  year={2010},
  publisher={Springer}
}

@book{adams2004knot,
  title={The knot book: an elementary introduction to the mathematical theory of knots},
  author={Adams, Colin Conrad},
  year={2004},
  publisher={American Mathematical Soc.}
}

@book{milnor1963morse,
  title={Morse theory},
  author={Milnor, John Willard},
  number={51},
  year={1963},
  publisher={Princeton university press}
}

@article{blair2013bridge,
  title={Bridge number and tangle products},
  author={Blair, Ryan},
  journal={Algebraic \& Geometric Topology},
  volume={13},
  number={2},
  pages={1125--1141},
  year={2013},
  publisher={Mathematical Sciences Publishers}
}

@phdthesis {atapourphdthesis,
    author = {Mahshid Atapour},
    title = {Topological Entanglement Complexity of Systems of Polygons and Walks in Tubes},
    year = {2008},
    school = {University of Saskatchewan},
    address = {Saskatoon, Canada}
}

@article{Atapour10,
	doi = {10.1142/s0218216510007760},
	year = 2010,
	publisher = {World Scientific Pub Co Pte Lt},
	volume = {19},
	number = {01},
	pages = {27--54},
	author = {M. Atapour and C. E. Soteros and C. Ernst and S. G. Whittington},
	title = {The linking probability for 2-component links which span a lattice tube},
	journal = {J. Knot Theory Ramifications}
}

@article{BES19,
	doi = {10.1088/1751-8121/ab0981},
	year = 2019,
	publisher = {{IOP} Publishing},
	volume = {52},
	number = {14},
	pages = {144003},
	author = {N R Beaton and J W Eng and C E Soteros},
	title = {Knotting statistics for polygons in lattice tubes},
	journal = {Journal of Physics A: Mathematical and Theoretical}
}

@article{COW20,
doi = {10.1088/1742-5468/ab7a20},
url = {https://doi.org/10.1088/1742-5468/ab7a20},
year = {2020},
month = {apr},
publisher = {IOP Publishing and SISSA},
volume = {2020},
number = {4},
pages = {043203},
author = {Caraglio, M and Orlandini, E and Whittington, S G},
title = {Translocation of links through a pore: effects of link complexity and size},
journal = {Journal of Statistical Mechanics: Theory and Experiment}
}

@article{COW17,
author = {Caraglio, M. and Orlandini, E. and Whittington, S. G.},
title = {Driven Translocation of Linked Ring Polymers through a Pore},
journal = {Macromolecules},
volume = {50},
number = {23},
pages = {9437-9444},
year = {2017},
doi = {10.1021/acs.macromol.7b02023},

URL = { 
    
        https://doi.org/10.1021/acs.macromol.7b02023
    
    

},
eprint = { 
    
        https://doi.org/10.1021/acs.macromol.7b02023
    
    

}

}

@article{RK23,
  title = {Nanopore translocation of topologically linked DNA catenanes},
  author = {Rheaume, Sierra N. and Klotz, Alexander R.},
  journal = {Phys. Rev. E},
  volume = {107},
  issue = {2},
  pages = {024504},
  numpages = {7},
  year = {2023},
  month = {Feb},
  publisher = {American Physical Society},
  doi = {10.1103/PhysRevE.107.024504},
  url = {https://link.aps.org/doi/10.1103/PhysRevE.107.024504}
}

@article{Madras99,
	author = {Madras, Neal},
    year= 1999,
    date= {1999/06/01},
    title = {A pattern theorem for lattice clusters},
    journal = {Annals of Combinatorics},
	pages = {357--384},
	volume = {3},
    number = {2},
	doi = {10.1007/BF01608793}
}

@article{Pip89,
	doi = {10.1016/0166-218x(89)90005-x},
	year = 1989,
	publisher = {Elsevier {BV}},
	volume = {25},
	number = {3},
	pages = {273--278},
	author = {Nicholas Pippenger},
	title = {Knots in random walks},
	journal = {Discrete Applied Mathematics}
}

@incollection {Sot98,
    AUTHOR = {Soteros, C. E.},
     TITLE = {Knots in graphs in subsets of {${\mathbb Z}^3$}},
 BOOKTITLE = {Topology and Geometry in Polymer Science ({M}inneapolis, {MN},
              1996)},
    SERIES = {IMA Vol. Math. Appl.},
    VOLUME = {103},
     PAGES = {101--133},
 PUBLISHER = {Springer, New York},
      YEAR = {1998},
   MRCLASS = {57M25 (05C10 82B41)},
  MRNUMBER = {1655040},
MRREVIEWER = {Stuart G. Whittington},
       DOI = {10.1007/978-1-4612-1712-1\_10},
}

@article{SSW92,
	doi = {10.1017/s0305004100075174},
	year = 1992,
	publisher = {Cambridge University Press ({CUP})},
	volume = {111},
	number = {1},
	pages = {75--91},
	author = {C. E. Soteros and D. W. Sumners and S. G. Whittington},
	title = {Entanglement complexity of graphs in Z3},
	journal = {Mathematical Proceedings of the Cambridge Philosophical Society}
}

@phdthesis{eng_phd_thesis,
    author = {Jeremy W. Eng},
    year = 2020,
    school = {University of Saskatchewan},
    address = {Saskatoon, Canada},
    title = {A Transfer Matrix Approach to Studying the Entanglement Complexity of Self-Avoiding Polygons in Lattice Tubes}
}

@article{rechnitzer_rensburg,
	doi = {10.1088/1742-5468/2011/09/p09008},
	year = 2011,
	publisher = {{IOP} Publishing},
	volume = {2011},
	number = {09},
	pages = {P09008},
	author = {Janse van Rensburg, E. J. and A Rechnitzer},
	title = {Minimal knotted polygons in cubic lattices},
	journal = {Journal of Statistical Mechanics: Theory and Experiment}
}

@article {Suma17,
	author = {Suma, Antonio and Micheletti, Cristian},
	title = {Pore translocation of knotted {DNA} rings},
	volume = {114},
	number = {15},
	pages = {E2991--E2997},
	year = {2017},
	doi = {10.1073/pnas.1701321114},
	publisher = {National Academy of Sciences},
	issn = {0027-8424},
	journal = {Proceedings of the National Academy of Sciences}
}

@article{friedl2014specious,
  title={A specious unlinking strategy},
  author={Friedl, Stefan and Nagel, Matthias and Powell, Mark},
  journal={Periodica Mathematica Hungarica},
  volume={69},
  number={2},
  pages={207--210},
  year={2014},
  publisher={Springer}
}

@article{Inprep2025,
  title={Link statistics for pairs of polygons spanning a tube or nano-channel},
  author={Eng, Jeremy and Pongtanapaisan, Puttipong and Scharein, Rob and Soteros, Chris},
  journal={In preparation},
  volume={},
  number={},
  pages={},
  year={2025}
}

@article{fox1962quick,
  title={A quick trip through knot theory},
  author={Fox, Ralph H},
  journal={Topology of 3-manifolds and related topics},
  pages={120--167},
  year={1962},
  publisher={Prentice-Hall}
}

@article{eisermann2007knot,
  title={Knot colouring polynomials},
  author={Eisermann, Michael},
  journal={Pacific Journal of Mathematics},
  volume={231},
  number={2},
  pages={305--336},
  year={2007},
  publisher={Mathematical Sciences Publishers}
}

@book{cromwell2004knots,
  title={Knots and links},
  author={Cromwell, Peter R},
  year={2004},
  publisher={Cambridge university press}
}

@article{kanenobu2003tangles,
  title={Tangles with up to seven crossings},
  author={Kanenobu, Taizo and Saito, Hirofusa and Satoh, Shin},
  journal={Interdisciplinary information sciences},
  volume={9},
  number={1},
  pages={127--140},
  year={2003},
  publisher={The Editorial Committee of the Interdisciplinary Information Sciences}
}

@article{KIM2025109353,
title = {Minimum lattice length of 2-bridge knots and links},
journal = {Topology and its Applications},
volume = {371},
pages = {109353},
year = {2025},
note = {Special issue dedicated to Iberoamerican and Pan Pacific International Conference on Topology and its Applications. Benemérita Universidad Autónoma de Puebla, Ciudad de Puebla, México, September 2023.},
issn = {0166-8641},
doi = {https://doi.org/10.1016/j.topol.2025.109353},
url = {https://www.sciencedirect.com/science/article/pii/S0166864125001518},
author = {Hyoungjun Kim}
}

@article{Ishihara_2012,
doi = {10.1088/1751-8113/45/6/065003},
url = {https://doi.org/10.1088/1751-8113/45/6/065003},
year = {2012},
month = {jan},
publisher = {IOP Publishing},
volume = {45},
number = {6},
pages = {065003},
author = {Ishihara, K and Scharein, R and Diao, Y and Arsuaga, J and Vazquez, M and Shimokawa, K},
title = {Bounds for the minimum step number of knots confined to slabs in the simple cubic lattice},
journal = {Journal of Physics A: Mathematical and Theoretical}
}

@article{Scharein_2009,
doi = {10.1088/1751-8113/42/47/475006},
url = {https://doi.org/10.1088/1751-8113/42/47/475006},
year = {2009},
month = {nov},
publisher = {},
volume = {42},
number = {47},
pages = {475006},
author = {Scharein, R and Ishihara, K and Arsuaga, J and Diao, Y and Shimokawa, K and Vazquez, M},
title = {Bounds for the minimum step number of knots in the simple cubic lattice},
journal = {Journal of Physics A: Mathematical and Theoretical}
}

@article{nakanishi1983prime,
  title={Prime and simple links},
  author={Nakanishi, Yasutaka},
  journal={Math. Sem. Notes Kobe Univ},
  volume={11},
  pages={249--256},
  year={1983}
}

@article{nakanishi1981primeness,
  title={Primeness of links},
  author={Nakanishi, Yasutaka},
  journal={Math. Sem. Notes Kobe Univ},
  volume={9},
  number={2},
  pages={415--440},
  year={1981}
}

\Addresses

\end{document}